\documentclass[10pt]{amsart}
\usepackage{amscd}
\usepackage{amssymb}
\usepackage[all]{xy}
\usepackage{lmodern} 
\usepackage[T1]{fontenc}
\date{18 February 2010}
\title{Central Extensions of Gerbes}
\newcommand{\hide}[1]{#1}
\author{Amnon Yekutieli}
\address{Department of  Mathematics
Ben Gurion University,
Be'er Sheva 84105,
Israel}
\email{amyekut@math.bgu.ac.il}
\thanks{
Primary: 20L05; Secondary: 18G50, 20J06, 18F20.}
\thanks{This research was supported by the US-Israel Binational
Science Foundation and the Israel Science Foundation.}

\newtheorem{thm}[equation]{Theorem}
\newtheorem{cor}[equation]{Corollary}
\newtheorem{prop}[equation]{Proposition}
\newtheorem{lem}[equation]{Lemma}
\theoremstyle{definition}
\newtheorem{dfn}[equation]{Definition}
\newtheorem{rem}[equation]{Remark}
\newtheorem{exa}[equation]{Example}

\newtheorem{que}[equation]{Question}

\newtheorem{construction}[equation]{Construction}

\numberwithin{equation}{section}

\newcommand{\iso}{\xrightarrow{\simeq}}

\newcommand{\xar}{\xrightarrow}
\newcommand{\opn}{\operatorname}
\newcommand{\cat}[1]{\operatorname{\mathsf{#1}}}

\newcommand{\rmitem}[1]{\item[\text{\textup{(#1)}}]}

\newcommand{\mcal}[1]{\mathcal{#1}}

\newcommand{\mrm}[1]{\mathrm{#1}}
\newcommand{\mbb}[1]{\mathbb{#1}}

\newcommand{\tup}[1]{\textup{#1}}
\newcommand{\bsym}[1]{\boldsymbol{#1}}

\newcommand{\til}[1]{\tilde{#1}}

\newcommand{\K}{\mbb{K}}

\newcommand{\N}{\mbb{N}}

\renewcommand{\d}{\mrm{d}}
\newcommand{\twoto}{\Rightarrow}
\newcommand{\twoiso}{\stackrel{\simeq}{\Longrightarrow}}
\newcommand{\twocong}{\stackrel{\simeq}{\Longleftrightarrow}}
\newcommand{\gerbe}[1]{\bsym{\mcal{#1}}}
\newcommand{\twocat}[1]{\operatorname{\bsym{\cat{#1}}}}
\newcommand{\twofun}[1]{\operatorname{\bsym{#1}}}

\begin{document}

\NoCompileMatrices

\begin{abstract}
We introduce the notion of central extension of gerbes on a topological space
$X$. We show that there are obstruction classes to lifting objects and
isomorphisms in a central extension. We also discuss pronilpotent gerbes.
These results are used in the paper \cite{Ye} to study twisted deformation
quantization on algebraic varieties. 
\end{abstract}

\maketitle
\tableofcontents

\setcounter{section}{-1}
\section{Introduction}
\label{sec.Int}

A {\em gerbe} $\gerbe{G}$ on a topological space $X$ is the geometric
version of a connected nonempty groupoid. Thus $\gerbe{G}$
associates a groupoid $\gerbe{G}(U)$ to any open set $U \subset X$,
and to any inclusion $V \subset U$
of open sets there is a restriction functor
$\gerbe{G}(U) \to \gerbe{G}(V)$.
These have to satisfy a lot of conditions (for the benefit of the reader we
have included a review in Section \ref{sec:prest}).
Gerbes arise in various contexts; but
for us the are mainly important as ``bookkeeping devices'' for certain
geometric data. At the end of the introduction we will outline the main
application we have in mind. 

A key question is to determine if a given gerbe $\gerbe{G}$ is
{\em trivial}, namely if $\opn{Ob} \gerbe{G}(X) \neq \emptyset$.
When $\gerbe{G}$ is {\em abelian}, with band some sheaf $\mcal{N}$
of abelian groups, there is an obstruction class in the \v{C}ech cohomology
group $\check{\mrm{H}}^2(X, \mcal{N})$ that vanishes if and only if  $\gerbe{G}$
is trivial. However for a {\em nonabelian} gerbe $\gerbe{G}$ there is no
useful obstruction theory, since the structure is too complicated.
There is Giraud's nonabelian cohomology theory \cite{Gi}, but that
does not provide an effective answer. 

We noticed during our work on deformation quantization that the
gerbes occurring there are {\em pronilpotent} (see explanation
below). Such gerbes are composed of central extensions, and for
those extensions we can construct useful obstruction classes. 

A {\em central extension of gerbes} on $X$ is a diagram
\begin{equation} \label{eqn:14}
1 \to  \mcal{N} \to \gerbe{G} \xar{F} \gerbe{H} \to 1 , 
\end{equation}
in which $\gerbe{G}$ and $\gerbe{H}$ are gerbes,  
$F : \gerbe{G} \to \gerbe{H}$ is a weak epimorphism of gerbes, and
$\mcal{N} = \opn{Ker}(F)$ is a sheaf of abelian groups in the center of 
$\gerbe{G}$.
This notion is technically quite complicated (see Definition \ref{dfn:2}), but
in principle it is just a generalization of the notion of central extension of
sheaves of groups
\begin{equation} \label{eqn:26}
1 \to \mcal{N} \to \mcal{G} \to \mcal{H} \to 1 .
\end{equation}

Consider a central extension of gerbes (\ref{eqn:14}). Suppose 
$i, j$ are two objects of $\gerbe{G}(X)$, and 
$h : F(i) \to F(j)$ is an isomorphism if $\gerbe{H}(X)$. 
Then there is obstruction class
\[ \opn{cl}^1_F(h) \in \check{\mrm{H}}^1(X, \mcal{N}) . \]
The first main result of the paper, Theorem \ref{thm.3}, says that 
$\opn{cl}^1_F(h)$ vanishes if and only if $h$ can be lifted to an
isomorphism $g : i \to j$ in $\gerbe{G}(X)$.

Given an object $j$ of $\gerbe{H}(X)$, 
we define (under some hypothesis) an obstruction class
\[ \opn{cl}^2_F(j) \in \check{\mrm{H}}^2(X, \mcal{N}) . \]
The second main result of the paper, Theorem \ref{thm.2}, says that $j$
lifts to an object of $\gerbe{G}(X)$ if and only if 
$\opn{cl}^2_F(j) = 1$. 

There are three typical situations where central extensions of gerbes occur.
The first is when we are given a central extension of sheaves of groups 
(\ref{eqn:26}). This is discussed briefly in Example \ref{exa:27}.

Another situation is when we take any gerbe $\gerbe{G}$, and look at 
$\mcal{N} := \mrm{Z}(\gerbe{G})$, the center of $\gerbe{G}$, which
is a sheaf of abelian groups. We get a central extension
\[ 1 \to  \mrm{Z}(\gerbe{G}) \to \gerbe{G} \xar{F} 
\gerbe{G} / \mrm{Z}(\gerbe{G}) \to 1 . \]
Global objects of $\gerbe{G} / \mrm{Z}(\gerbe{G})$ are called
{\em fake global objects of $\gerbe{G}$}. See Section \ref{sec.fake}. 

The third situation, which is the most important for us, is when the
gerbe $\gerbe{G}$ is pronilpotent,
i.e.\ it is complete with respect to a central filtration
$\{ \gerbe{N}_p \}_{p \in \mbb{N}}$
(see Definition \ref{dfn.12}). 
Then for any $p$ there is a central extension of gerbes
\[ 1 \to \gerbe{N}_p / \gerbe{N}_{p+1}
\to \gerbe{G} / \gerbe{N}_{p+1} \to
\gerbe{G} / \gerbe{N}_p \to 1 . \]
The obstruction classes can detect whether the groupoid
$(\gerbe{G} / \gerbe{N}_p)(X)$ is nonempty or connected
for any $p$; but passing to the limit is more delicate. This is done in the
third main result of the paper, namely Theorem \ref{thm.4}. 

Presumably our results can be extended, with minor changes, to sites
other than a topological space (e.g.\ the \'etale site of a scheme).
But we did not explore this direction. 

Here is an outline of the role gerbes have in our paper \cite{Ye}. 
Suppose $X$ is a smooth algebraic variety over a
field $\K$ of characteristic $0$. We are interested in {\em twisted
deformations} of $\mcal{O}_X$. A twisted (associative or
Poisson) deformation $\gerbe{A}$ is a collection of locally defined 
(associative or Poisson) deformations $\mcal{A}_i$ of $\mcal{O}_X$, together
with a collection of locally defined gauge equivalences 
$\mcal{A}_i \iso \mcal{A}_j$ between them. The bookkeeping data of
deformations and gauge equivalences are encoded in the {\em gauge gerbe} 
$\gerbe{G}$ of $\gerbe{A}$. Here is just a hint of how this goes --
see Remark \ref{rem.1} for a few more details, or the paper \cite{Ye} for the
full story. Let $U \subset X$ be an open set. 
Then to any object $i$ in the groupoid $\gerbe{G}(U)$ we
attach a deformation 
$\mcal{A}_i$ of $\mcal{O}_{U}$; and to any morphism 
$g : i \to j$ in $\gerbe{G}(U)$ we attach a gauge equivalence
$\gerbe{A}(g) : \mcal{A}_i \iso \mcal{A}_j$.
Thus the groupoid $\gerbe{G}(X)$ carries the
information of global deformations: objects of 
$\gerbe{G}(X)$ correspond to global deformations of $\mcal{O}_X$
belonging to $\gerbe{A}$, and isomorphic objects correspond to gauge
equivalent deformations. Since the gauge gerbe $\gerbe{G}$ is 
pronilpotent, we can often use Theorems \ref{thm.3}, \ref{thm.2} and
\ref{thm.4} to figure out how many connected components the
groupoid $\gerbe{G}(X)$ has.

\medskip \noindent
\textbf{Acknowledgments.}
Work on this paper began together with Fredrick Leitner, and I wish to thank
him for his contributions, without which the paper could not have been written. 
Thanks also to Lawrence Breen for reading an early version of the
paper and offering valuable suggestions. Finally I wish to thank the referee
for his/her input, and especially for discovering a subtle error in one of the
main results, as it was stated in a previous version of the paper.

%\cleardoublepage
\section{Recalling some Facts on $2$-Categories}
\label{sec.recall}

There are several sources in the literature on $2$-categories and prestacks,
e.g.\ \cite{Be}, \cite{Gi}, \cite{Ma}, \cite{Mo}, \cite{Le}, \cite{KS}
and \cite{Br}. Unfortunately
there is disagreement on terminology among the sources, and hence we
feel it is
better to start with an exposition of the conventions we adopted, and a
recollection some facts. 

First we must establish some set-theoretical background, in order to
avoid paradoxical phenomena. Recall that in set theory all
mathematical objects and operations are interpreted as sets, with
suitable additional properties. Following \cite{Ma}
we fix a {\em Grothendieck universe} $\cat{U}$, which is a set closed
under standard set-theoretical operations, and large enough such
that the objects of interest for us (e.g.\ the topological space $X$ 
in Section \ref{sec:prest}) are elements of
$\cat{U}$. We refer to elements of $\cat{U}$ as {\em small sets}.
A category $\cat{C}$ such that
$\opn{Ob}(\cat{C}) \in \cat{U}$,
and 
$\opn{Hom}_{\cat{C}}(C_0, C_1) \in \cat{U}$
for every pair $C_0, C_1 \in \opn{Ob}(\cat{C})$,
is called a {\em small category}.

By $\cat{Set}$ we refer the category of small sets; thus in effect
$\opn{Ob}(\cat{Set}) = \cat{U}$.
Likewise $\cat{Grp}$, $\cat{Mod} A$ etc.\ refer to the categories of
small groups, small $A$-modules (over a small ring $A$) etc.\ 
A category $\cat{C}$ such that 
$\opn{Ob}(\cat{C}) \subset \cat{U}$,
and $\opn{Hom}_{\cat{C}}(C_0, C_1) \in \cat{U}$
for every pair $C_0, C_1 \in \opn{Ob}(\cat{C})$,
is called a $\cat{U}$-category.
Thus $\cat{Set}$ is a $\cat{U}$-category, but it is not small. 

Next we introduce a bigger universe $\cat{V}$, such that 
$\cat{U} \in \cat{V}$. Then 
$\opn{Ob}(\cat{Set}),$ \linebreak 
$\opn{Ob}(\cat{Grp}),\ldots \in \cat{V}$.
In order to distinguish between them, we call $\cat{U}$ the small
universe, and $\cat{V}$ is the large universe. 
The set of all $\cat{U}$-categories is denoted by
$\twocat{Cat}$. Note that 
$\twocat{Cat}$ is a $\cat{V}$-category, but (this is the whole
point!) it is not a $\cat{U}$-category.

By default sets, groups etc.\ will be assumed to be small; and
categories will be assumed to be $\cat{U}$-categories. 

A {\em $2$-category} $\twocat{C}$ is a ``category enriched in
categories''. (Some authors use the term ``strict $2$-category''.)
This means the following. There is a set $\opn{Ob} (\twocat{C})$, whose
elements are called {\em objects} of $\twocat{C}$. For any pair of objects
$\cat{C}_0, \cat{C}_1 \in \opn{Ob} (\twocat{C})$
there is a category 
$\twocat{C}(\cat{C}_0, \cat{C}_1)$.
The objects of the category $\twocat{C}(\cat{C}_0, \cat{C}_1)$
are called {\em $1$-morphisms}, and the morphisms of 
$\twocat{C}(\cat{C}_0, \cat{C}_1)$ are called {\em $2$-morphisms}.
For every triple 
$\cat{C}_0, \cat{C}_1, \cat{C}_2 \in \opn{Ob} (\twocat{C})$
there is a bifunctor
\[ \twocat{C}(\cat{C}_0, \cat{C}_1) \times
\twocat{C}(\cat{C}_1, \cat{C}_2) \to
\twocat{C}(\cat{C}_0, \cat{C}_2) , \]
called {\em horizontal composition}.
Horizontal composition has to be associative (as bifunctor).
For any $\cat{C} \in \opn{Ob} (\twocat{C})$ there is a distinguished
$1$-morphism
$\bsym{1}_{\cat{C}} \in \opn{Ob} \bigl( \twocat{C}(\cat{C}, \cat{C}) \bigr)$,
called the identity $1$-morphism of $\cat{C}$. Horizontal composition with
$\bsym{1}_{\cat{C}}$, on either side, is 
required to be the identity functor. 

Given a $1$-morphism 
$F \in \opn{Ob} \bigl( \twocat{C}(\cat{C}_0, \cat{C}_1) \bigr)$,
we write $F : \cat{C}_0 \to \cat{C}_1$. 
The notation for horizontal composition is $\circ$; so given
$1$-morphisms 
$F_1 : \cat{C}_0 \to \cat{C}_1$ and
$F_2 : \cat{C}_1 \to \cat{C}_2$, their composition is
$F_2 \circ F_1 : \cat{C}_0 \to \cat{C}_2$.
We sometimes denote the set 
$\opn{Ob} \bigl( \twocat{C}(\cat{C}_0, \cat{C}_1) \bigr)$
of $1$-morphisms $\cat{C}_0 \to \cat{C}_1$  by
$\opn{Hom}_{\twocat{C}}(\cat{C}_0, \cat{C}_1)$.

Let 
$F, G \in \opn{Ob} \bigl( \twocat{C}(\cat{C}_0, \cat{C}_1) \bigr)$,
and let
$\eta \in \opn{Hom}_{\twocat{C}(\cat{C}_0, \cat{C}_1)}(F, G)$;
i.e.\ $\eta$ is a $2$-morphism. We write 
$\eta: F \Rightarrow G$.
This data is usually depicted as a diagram:
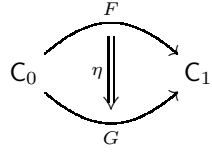
\hide{ \[ \UseTips
\xymatrix  @C=5ex @R=1.5ex  {
& {}
\ar@{=>}[dd]_{\eta}
\\
\cat{C}_0
\ar@(ur,ul)[rr]^{F}
\ar@(dr,dl)[rr]_{G}
& & \cat{C}_1
\\
& {}
} \] }
The composition rule in the category
$\twocat{C}(\cat{C}_0, \cat{C}_1)$
is called {\em vertical composition}, and we denote it by $*$.
Thus if $H \in \opn{Ob} \bigl( \twocat{C}(\cat{C}_0, \cat{C}_1) \bigr)$ 
is another $1$-morphism, and $\zeta : G \Rightarrow H$ is a
$2$-morphism, then
by vertical composition we get
$\zeta * \eta : F \Rightarrow H$.
\hide{ \[ \UseTips
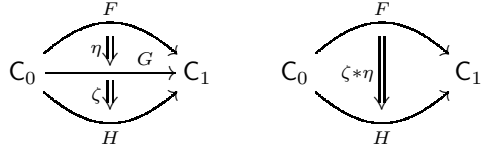
\xymatrix  @C=5ex @R=1.5ex  {
& {}
\ar@{=>}[d]_{\eta}
\\
\cat{C}_0
\ar@(ur,ul)[rr]^{F}
\ar[rr]^(0.7){G}
\ar@(dr,dl)[rr]_{H}
& \ar@{=>}[d]_{\zeta}
& \cat{C}_1
\\
& {}
} \qquad
\xymatrix  @C=5ex @R=1.5ex  {
& {}
\ar@{=>}[dd]_{\zeta * \eta}
\\
\cat{C}_0
\ar@(ur,ul)[rr]^{F}
\ar@(dr,dl)[rr]_{H}
& & \cat{C}_1
\\
& {}
} \] }
Let us denote by $\bsym{1}_{F}$ the identity automorphism of the
object $F$ in
the category
$\twocat{C}(\cat{C}_0, \cat{C}_1)$.
Then $\bsym{1}_{F} * \eta = \eta = \eta * \bsym{1}_{G}$. 

The pictorial description of horizontal composition is this:
given a diagram 
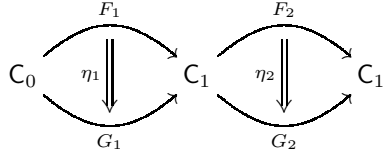
\hide{ \[ \UseTips
\xymatrix  @C=5ex @R=1.5ex  {
& {}
\ar@{=>}[dd]_{\eta_1}
& & 
\ar@{=>}[dd]_{\eta_2}
\\
\cat{C}_0
\ar@(ur,ul)[rr]^{F_1}
\ar@(dr,dl)[rr]_{G_1}
& & \cat{C}_1
\ar@(ur,ul)[rr]^{F_2}
\ar@(dr,dl)[rr]_{G_2}
& & \cat{C}_1
\\
& & & {}
} \] }
the horizontal composition of the $2$-morphisms $\eta_1$ and $\eta_2$ is
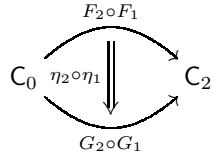
\hide{ \[ \UseTips
\xymatrix  @C=5ex @R=1.5ex  {
& {}
\ar@{=>}[dd]_{\eta_2 \circ \eta_1}
\\
\cat{C}_0
\ar@(ur,ul)[rr]^{F_2 \circ F_1}
\ar@(dr,dl)[rr]_{G_2 \circ G_1}
& & \cat{C}_2
\\
& {}
} \] }

The horizontal and vertical compositions are required to satisfy the following
condition, called the {\em exchange condition}. Suppose we are
given a diagram 
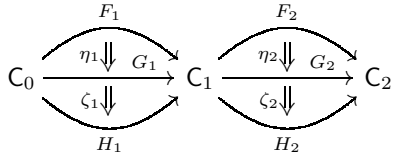
\hide{ \[ \UseTips
\xymatrix  @C=5ex @R=1.5ex  {
& {}
\ar@{=>}[d]_{\eta_1}
& {}
& 
\ar@{=>}[d]_{\eta_2}
\\
\cat{C}_0
\ar@(ur,ul)[rr]^{F_1}
\ar[rr]^(0.7){G_1}
\ar@(dr,dl)[rr]_{H_1}
& \ar@{=>}[d]_{\zeta_1} 
& {}
\cat{C}_1
\ar@(ur,ul)[rr]^{F_2}
\ar[rr]^(0.7){G_2}
\ar@(dr,dl)[rr]_{H_2}
& \ar@{=>}[d]_{\zeta_2} 
& {}
\cat{C}_2
\\
& {}
& {}
& {}
} \] }
in $\twocat{C}$. Then 
\[ (\zeta_2 \circ \zeta_1) * (\eta_2 \circ \eta_1) = 
(\zeta_2 * \eta_2) \circ (\zeta_1 * \eta_1) \]
as $2$-morphisms 
$F_2 \circ F_1 \Rightarrow H_2 \circ H_1$.

Regarding set-theoretical issues, we require that 
$\opn{Ob}(\twocat{C}) \subset \cat{V}$,
$\opn{Ob} \bigl( \twocat{C}(\cat{C}_0, \cat{C}_1) \bigr) \in \cat{V}$,
and
$\opn{Hom}_{\twocat{C}(\cat{C}_0, \cat{C}_1)}(F, G) \in \cat{U}$.
Note that if we forget the $2$-morphisms in $\twocat{C}$, then
$\twocat{C}$ becomes a $\cat{V}$-category.

The basic example of a $2$-category is this:

\begin{exa} \label{exa.1} 
The $2$-category of $\cat{U}$-categories, denoted by $\bsym{\cat{Cat}}$.
The set $\opn{Ob} (\bsym{\cat{Cat}})$ of all $\cat{U}$-categories was already
mentioned. The $1$-morphisms in
$\bsym{\cat{Cat}}(\cat{C}_0, \cat{C}_1)$
are the functors $F : \cat{C}_0 \to \cat{C}_1$ between these categories.
And the $2$-morphisms $\eta : F \Rightarrow G$ are the natural
transformations. The composition rules are the usual ones.
\end{exa}

Here is another example, of a different flavor. 

\begin{exa}
Let $\K$ be a commutative ring. The category
$\cat{DGMod} \K$ of DG (differential graded) $\K$-modules can be made into a 
$2$-category, as follows. 
Given $M, N \in \opn{Ob} (\cat{DGMod} \K)$, the 
$1$-morphisms $F : M \to N$ are the usual morphisms in 
$\cat{DGMod} \K$, i.e.\ $\K$-linear DG module homomorphisms. 
Now such a homomorphism $F : M \to N$ can be viewed as a $0$-cocycle
in the DG module $\opn{Hom}_{\K}(M, N)$. Given another such homomorphism
$G : M \to N$, the $2$-morphisms $\eta : F \Rightarrow G$ are by definition 
the $0$-coboundaries $\eta \in \opn{Hom}_{\K}(M, N)$ such that 
$G = \eta + F$. Compositions are obvious.
\end{exa}

Suppose $F, G \in \opn{Hom}_{\twocat{C}}(\cat{C}_0, \cat{C}_1)$. We
say
that $F$ and $G$ are {\em $2$-isomorphic} if there is some
$2$-isomorphism
$\eta : F \twoiso G$ in the category 
$\opn{Hom}_{\twocat{C}}(\cat{C}_0, \cat{C}_1)$; 
we denote this by $F \twocong G$. 
A diagram (of $1$-morphisms)
\hide{ \[ \UseTips \xymatrix @C=5ex @R=5ex {
\cat{C}_0
\ar[r]^{D}
\ar[dr]_{F}
& 
\cat{C}_1
\ar[d]^{E}
\\
& 
\cat{C}_2
}  \] }
is {\em commutative up to $2$-isomorphism} if 
$E \circ D \twocong F$. 

There is an intrinsic notion of {\em equivalence} in a $2$-category
$\twocat{C}$.
A $1$-morphism $F : \cat{C} \to \cat{D}$ is an called an equivalence
if there
is a $1$-morphism $G : \cat{D} \to \cat{C}$
such that $G \circ F \twocong \bsym{1}_{\cat{C}}$ and 
$F \circ G \twocong \bsym{1}_{\cat{D}}$.
This generalizes the usual notion of equivalence (of categories) in
Example \ref{exa.1}.

Suppose $\twocat{C}$ and $\twocat{D}$ are $2$-categories. A 
{\em $2$-functor} $\twofun{F} : \twocat{C} \to \twocat{D}$
is a triple 
$\bsym{F} = (F_0, F_1, F_2)$, 
consisting of functions of the following kinds. The function $F_0$, 
called the $0$-component of $\twofun{F}$,
assigns to each object
$\cat{C} \in \opn{Ob} \twocat{C}$, an object
$F_0(\cat{C}) \in \opn{Ob} \twocat{D}$. 
The function $F_1$ assigns to each morphism
$G : \cat{C}_0 \to \cat{C}_1$ in $\twocat{C}$, a $1$-morphism
\[ F_1(G) : F_0 (\cat{C}_0) \to F_0 (\cat{C}_1)  \]
in $\twocat{D}$. And the function $F_2$  assigns to each $2$-morphism
$\eta : G \twoto G'$ in $\twocat{C}$, a $2$-morphism 
\[ F_2(\eta) : F_1(G) \twoto F_1(G') \]
in $\twocat{D}$. The condition is that the functions 
$(F_0, F_1,F_2)$ preserve compositions and units. Thus, if we forget
$2$-morphisms, the pair $(F_0, F_1)$ is a functor
\[ (F_0, F_1) : \twocat{C} \to \twocat{D} \]
between these categories. And for every 
$\cat{C}_0, \cat{C}_1 \in \opn{Ob} \twocat{C}$, 
the pair $(F_1, F_2)$ is a functor
\[ (F_1, F_2) : \opn{Hom}_{\twocat{C}}(\cat{C}_0, \cat{C}_1) \to 
\opn{Hom}_{\twocat{D}} \bigl( F_0(\cat{C}_0), F_0(\cat{C}_1) \bigr) .
\]

Let $\twocat{C}$ and $\twocat{D}$ be $2$-categories, and let
$\twofun{F}, \twofun{G} : \twocat{C} \to \twocat{D}$
be $2$-functors, with components 
\[ \twofun{F} = (F_0, F_1, F_2)\, , \, \twofun{G} = (G_0, G_1, G_2) . \]
A $1$-morphism (sometimes called a $2$-natural
transformation) 
$\twofun{p} : \twofun{F} \to \twofun{G}$ 
is a function that assigns to each $\cat{C} \in \opn{Ob} \twocat{C}$
a $1$-morphism 
\[ \twofun{p}_{\cat{C}} : F_0(\cat{C}) \to  G_0(\cat{C}) \]
in $\twocat{D}$, such that for every 
$E \in \opn{Hom}_{\twocat{C}}(\cat{C}_0, \cat{C}_1)$
one has
\[ \twofun{p}_{\cat{C}_1} \circ F_1(E) =
G_1(E) \circ \twofun{p}_{\cat{C}_0}  \]
in 
$\opn{Hom}_{\twocat{D}} \bigl( F_0(\cat{C}_0), G_0(\cat{C}_1) \bigr)$.
Given another $2$-functor $\twofun{H} : \twocat{C} \to \twocat{D}$, 
and a $1$-morphism 
$\twofun{q} : \twofun{G} \to \twofun{H}$, 
the composition 
$\twofun{q} \circ \twofun{p} : \twofun{F} \to \twofun{H}$
is defined in the obvious way.

Now suppose 
$\twofun{p}, \twofun{q} : \twofun{F} \to \twofun{G}$
are $1$-morphisms between $2$-functors
$\twofun{F} , \twofun{G} : \twocat{C} \to \twocat{D}$ as above.
A $2$-morphism $\twofun{\eta} : \twofun{p} \to \twofun{q}$
(sometimes called a modification) is a function that assigns to each 
$\cat{C} \in \opn{Ob} \twocat{C}$, a $2$-morphism
$\twofun{\eta}_{\cat{C}} : \twofun{p}_{\cat{C}} \twoto
\twofun{q}_{\cat{C}}$
in 
$\opn{Hom}_{\twocat{D}} \bigl( F_0(\cat{C}), G_0(\cat{C}) \bigr)$.
The condition is that 
\[ \twofun{\eta}_{\cat{C}_1} \circ \twofun{p}_{\cat{C}_1} \circ
\twofun{F} = 
\twofun{\eta}_{\cat{C}_0} \circ \twofun{q}_{\cat{C}_0} \circ
\twofun{G} , \]
as functions
\[ \opn{Hom}_{\twocat{C}}(\cat{C}_0, \cat{C}_1)  \to 
\opn{Hom}_{\twocat{D}} \bigl( F_0(\cat{C}_0), G_0(\cat{C}_1) \bigr) .
\]
If $\twofun{r} : \twofun{F} \to \twofun{G}$
is yet another $1$-morphism, and 
$\twofun{\zeta} : \twofun{q} \to \twofun{r}$
is a $2$-morphism, then the composition 
$\twofun{\zeta} * \twofun{\eta} : \twofun{p} \to \twofun{r}$ is
defined in the obvious
way. We say that the $2$-morphism 
$\twofun{\eta} : \twofun{p} \to \twofun{q}$
is a $2$-isomorphism if each $\twofun{\eta}_{\cat{C}}$ is a
$2$-isomorphism. 

A $2$-functor $\twofun{F} : \twocat{C} \to \twocat{D}$ 
is called a {\em $2$-equivalence} if there is a $2$-functor
$\twofun{G} : \twocat{D} \to \twocat{C}$, and $2$-isomorphisms
$\twofun{G} \circ \twofun{F} \twoiso \twofun{1}_{\twocat{C}}$
and 
$\twofun{F} \circ \twofun{G} \twoiso \twofun{1}_{\twocat{D}}$.
If a $2$-equivalence $\twocat{C} \to \twocat{D}$ exists, then we say
that $\twocat{C}$ and $\twocat{D}$ are $2$-equivalent.

One could make the set of all $2$-categories, with the operations defined above,
into a $2$-category, but that would take us outside of the large universe
$\cat{V}$. Therefore we shall be careful to consider only ``small collections''
of $2$-categories in this paper.

We shall also need to recall what are {\em pseudofunctors}
(sometimes called normalized pseudofunctors, or morphisms of
bicategories) from a category 
$\cat{N}$ to a $2$-category $\twocat{C}$.
A pseudofunctor $\bsym{F} : \cat{N} \to \twocat{C}$
is a triple $\bsym{F} = (F_0, F_1, F_2)$, consisting of functions of
the following kinds. 
The function $F_0$, called the $0$-component of $\bsym{F}$, assigns to
each object 
$N \in \opn{Ob} \cat{N}$ an object $F_0 (N) \in \opn{Ob} \twocat{C}$.
The function $F_1$ assigns to each morphism
$f : N_0 \to N_1$ in $\cat{N}$ a $1$-morphism
\[ F_1(f) : F_0 (N_0) \to F_0 (N_1)  \]
in $\twocat{C}$.
And the function $F_2$  assigns to each composable pair of morphisms
\[ N_0 \xar{f_1} N_1 \xar{f_2} N_2 \] 
in $\cat{N}$, a $2$-isomorphism
\[ F_2(f_1, f_2) : F_1(f_2) \circ F_1(f_1) \twoiso F_1(f_2 \circ f_1) 
\]
in $\twocat{C}$.
Here are the conditions. First, 
\begin{equation} \label{eqn:12}
F_2(f_2 \circ f_1, f_3) * F_2(f_1, f_2)  = 
F_2(f_1, f_3 \circ f_2) * F_2(f_2 , f_3) 
\end{equation}
for any composable triple 
\[ N_0 \xar{f_1} N_1 \xar{f_2} N_2 \xar{f_3} N_3 \]
of morphisms in $\cat{N}$.
Next, for any object $N \in \cat{N}$, with identity morphism 
$\bsym{1}_N$, it is required that
$F_1(\bsym{1}_N) = \bsym{1}_{F_0(N)}$, the identity $1$-morphism of
$F_0(N)$.
And lastly, the $2$-isomorphisms 
\[ F_2( \bsym{1}_{N_0}, f_1 ) : F_1(f_1) \circ \bsym{1}_{F_0(N_0)}
\twoiso
F_1(f_1) \]
and
\[ F_2( f_1, \bsym{1}_{N_1} ) : \bsym{1}_{F_0(N_1)} \circ F_1(f_1)
\twoiso
F_1(f_1) \]
have to be the identity $2$-automorphism of the $1$-morphism
$F_1(f_1)$.

The final abstract $2$-categorical fact that we need is that given a small
category $\cat{N}$ and a $2$-category $\twocat{C}$, the set
of all pseudofunctors $\bsym{F} : \cat{N} \to \twocat{C}$
is itself a $2$-category. The $1$-morphisms are defined as follows. 
Suppose $\bsym{F}, \bsym{G} : \cat{N} \to \twocat{C}$ are
pseudofunctors, with
components $\bsym{F} = (F_0, F_1, F_2)$ and $\bsym{G} = (G_0, G_1,
G_2)$. 
A $1$-morphism 
$\bsym{p} : \bsym{F} \to \bsym{G}$ is a pair 
$\bsym{p} = (p_1, p_2)$, whose $1$-component $p_1$ is a function
assigning
to any object $N \in \opn{Ob} \cat{N}$ a $1$-morphism 
\[ p_1(N) : F_0(N) \to G_0(N) \]
in $\twocat{C}$; and the $2$-component $p_2$ is a function assigning
to any
morphism $f : N_0 \to N_1$ in $\cat{N}$ a $2$-isomorphism
\[ p_2(f) : p_1(N_1) \circ F_1(f) \twoiso G_1(f) \circ p_1(N_0) \]
in $\twocat{C}$.
These are required to satisfy
the condition
\begin{equation} \label{eqn:8}
p_2(f_2 * f_1) * F_2(f_1, f_2) = 
G_2(f_1, f_2) * p_2(f_1) * p_2(f_2) 
\end{equation}
in $\opn{Hom}_{\twocat{C}} \bigl( F_0(N_0), G_0(N_1) \bigr)$,
for any composable pair of morphisms
$N_0 \xar{f_1} N_1 \xar{f_2} N_2$ in $\cat{N}$.

Horizontal composition of $1$-morphisms is defined as follows. Suppose
$\bsym{H} : \cat{N} \to \twocat{C}$ is another pseudofunctor, and
$\bsym{q} : \bsym{G} \to \bsym{H}$ is a $1$-morphism.
Their components are $\bsym{H} = (H_0, H_1, H_2)$ and
$\bsym{q}=  (q_1, q_2)$. 
Let
\[ r_1(N) : F_0(N) \to H_0(N) \]
be the $1$-morphism
\[ r_1(N) := q_1(N) \circ p_1(N) , \]
and let 
\[ r_2(f) : r_1(N_1) \circ F_1(f) \twoto
H_1(f) \circ r_1(N_0) \]
be the $2$-morphism
\[ r_2(f) := q_2(f) * p_2(f) . \]
Then we define the $1$-morphism
\[ \bsym{q} \circ \bsym{p} : \bsym{F} \to \bsym{H} \]
to be
\[ \bsym{q} \circ \bsym{p} := (r_1, r_2) . \]

Next consider $1$-morphisms 
$\bsym{p}, \bsym{q} : \bsym{F} \to \bsym{G}$.
A $2$-morphism $\bsym{\eta} : \bsym{p} \twoto \bsym{q}$ 
has only a $2$-component $\eta_2$, which is a function that
assigns to
each object $N \in \opn{Ob} \cat{N}$ a $2$-morphism
\[ \eta_2(N) : p_1(N) \twoto q_1(N) \]
in $\twocat{C}$. The condition is that
\[ q_2(f) * \eta_2(N_0) = \eta_2(N_1) * p_2(f) \]
for any $f : N_0 \to N_1$ in $\cat{N}$.
Given yet another $1$-morphism $\bsym{r} : \bsym{F} \to \bsym{G}$,
and a $2$-morphism $\bsym{\zeta} = (\zeta_2) : \bsym{q} \twoto \bsym{r}$, the
vertical composition 
$\bsym{\theta} := \bsym{\zeta} * \bsym{\eta} : \bsym{p} \twoto \bsym{r}$ 
has $2$-component
\[ \theta_2(N) := \zeta_2(N) * \eta_2(N) . \]

%\cleardoublepage
\section{Prestacks on a Topological Space}
\label{sec:prest}

Let $X$ be a topological space. 
We need some notation for open coverings. Let $U \subset X$ be an open set, and
let $\bsym{U} = \{ U_k \}_{k \in K}$ be an open covering of $U$, i.e.\ 
$U = \bigcup_{k \in K} U_k$. Given 
$k_0, \ldots, k_m \in K$ we write
\[ U_{k_0, \ldots, k_m} := U_{i_0} \cap \cdots \cap U_{k_m} . \]

Let $\mcal{S}$ be a sheaf of sets on $X$. For an open set $U \subset X$ we
denote by 
$\mcal{S}(U) = \Gamma(U, \mcal{S})$
the set of sections of $\mcal{S}$ on $U$.

Recall that a {\em prestack} $\gerbe{G}$ on $X$
is the geometrization of the notion of category, 
in the same way that a presheaf of sets is the geometrization of the
notion of a set. Formally speaking a prestack $\gerbe{G}$ is a 
pseudofunctor
\[ \gerbe{G} 
= (\gerbe{G}_0, \gerbe{G}_1, \gerbe{G}_2) :
(\cat{Open} X)^{\mrm{op}} \to 
\bsym{\cat{Cat}} , \]
where $\cat{Open} X$ is the small category whose objects are the open sets 
$U \subset X$, and the morphisms $V \to U$ are the inclusions
$V \subset U$. However we shall make things more explicit here, and introduce
some notation, to emphasize the geometry. 

Thus a prestack $\gerbe{G}$ 
on $X$ has the following structure. 
For any open set $U \subset X$ there is a category
$\gerbe{G}(U) := \gerbe{G}_0(U)$. 
Elements of the set $\opn{Ob} \gerbe{G}(U)$ shall be denoted by
the letters $i, j$ etc.; this is because we want to view them as indices.
We write
\begin{equation}
\gerbe{G}(U)(i,j) := \opn{Hom}_{\gerbe{G}(U)} (i, j) , 
\end{equation}
the set of morphisms in the category $\gerbe{G}(U)$ from $i$ to $j$.

There are restriction functors ($1$-morphisms $\twocat{Cat}$)
\[ \opn{rest}^{\gerbe{G}}_{U_1 / U_0} := \gerbe{G}_1(U_1 \to
U_0) : \gerbe{G}(U_0) \to
\gerbe{G}(U_1) \] 
for any inclusion $U_1 \subset U_0$ of open sets. 
And there are composition isomorphisms ($2$-isomorphisms in
$\twocat{Cat}$)
\[ \gamma^{\gerbe{G}}_{U_2 / U_1 / U_0} := 
\gerbe{G}_2(U_2 \to U_1 \to U_0) 
: \opn{rest}^{\gerbe{G}}_{U_2 / U_1} \circ 
\opn{rest}^{\gerbe{G}}_{U_1 / U_0} \twoiso
\opn{rest}^{\gerbe{G}}_{U_2 / U_0} \]
for a double inclusion $U_2 \subset U_1 \subset U_0$. 
Condition (\ref{eqn:12}) now becomes 
\begin{equation} \label{eqn:6} 
\gamma^{\gerbe{G}}_{U_3 / U_2 / U_0} * 
\gamma^{\gerbe{G}}_{U_2 / U_1 / U_0} = 
\gamma^{\gerbe{G}}_{U_3 / U_1 / U_0} *
\gamma^{\gerbe{G}}_{U_3 / U_2 / U_1}
\end{equation}
for a triple inclusion $U_3 \subset U_2 \subset U_1 \subset U_0$. 
And there are corresponding conditions for $U \xar{=} U$.

As explained in Section \ref{sec.recall}, the set of prestacks on
$X$ has a structure of $2$-category, which we denote by
$\bsym{\cat{PreStack}}\, X$. Again, we want to be more specific. 
Suppose $\gerbe{G}$ and $\gerbe{H}$ are two prestacks on
$X$. A {\em morphism of prestacks}
$F : \gerbe{G} \to \gerbe{H}$
is a $1$-morphism between these pseudofunctors. Thus there is a
functor
\[ F(U) : \gerbe{G}(U) \to \gerbe{H}(U) \]
for any open set $U$, together with an isomorphism of functors
\[ \psi^F_{U_1 / U_0} : F(U_1) \circ 
\opn{rest}^{\gerbe{G}}_{U_1 / U_0} \twoiso 
\opn{rest}^{\gerbe{H}}_{U_1 / U_0} \circ \, F(U_0) \]
for any inclusion $U_1 \subset U_0$ of open sets. These isomorphisms
are required to satisfy condition 
\[ \psi^F_{U_2 / U_0} * \gamma^{\gerbe{G}}_{U_2 / U_1 / U_0} =
\gamma^{\gerbe{H}}_{U_2 / U_1 / U_0} * 
\psi^F_{U_2 / U_1}  * \psi^F_{U_1 / U_0} \]
for a double inclusion $U_2 \subset U_1 \subset U_0$.

The composition of morphisms of prestacks
$\gerbe{G} \xar{F} \gerbe{H} \xar{E} \gerbe{K}$
is denoted by $E \circ F$. 

Suppose $D, E, F : \gerbe{G} \to \gerbe{H}$
are morphisms between prestacks. We will denote $2$-morphisms between $E$
and $F$ by $\eta : E \twoto F$.
And the (vertical) composition with a $2$-morphism
$\zeta : D \twoto E$ is denoted by $\eta * \zeta : D \twoto F$. 

As in any $2$-category, we can say when a morphism of prestacks
$F : \gerbe{G} \to \gerbe{H}$ (i.e.\ a $1$-morphism
in $\bsym{\cat{PreStack}}\, X$) is an equivalence. 
This just means that there is a morphism of prestacks
$E : \gerbe{H} \to \gerbe{G}$,
and $2$-isomorphisms 
$E \circ F \twoiso \bsym{1}_{\gerbe{G}}$
and
$F \circ E \twoiso \bsym{1}_{\gerbe{H}}$. 
But here there is also a geometric characterization: $F$ is an equivalence if
and only if for any open set $U \subset X$ the functor
$F(U) : \gerbe{G}(U) \to \gerbe{H}(U)$
is an equivalence. 

Let $\gerbe{G}$ be a prestack on $X$. By a {\em subprestack} of $\gerbe{G}$ we
mean a prestack $\gerbe{N}$ such that $\gerbe{N}(U)$ is a subcategory of 
$\gerbe{G}(U)$ for every open set $U$, and such that the restriction functors
$\mrm{rest}^{\gerbe{N}}_{-/-}$ and the composition isomorphisms
$\gamma^{\gerbe{N}}_{- / - / -}$ are the same as those of $\gerbe{G}$.

Suppose $\gerbe{G}$ is a prestack on $X$. Take an open set $U \subset X$
and two objects $i, j \in \opn{Ob} \gerbe{G}(U)$. There is a
presheaf of sets $\gerbe{G}(i,j)$ on $U$, called the {\em presheaf of
morphisms}, defined as follows.
For an open set $V \subset U$ we define the set
\[ \gerbe{G}(i,j)(V) := 
\opn{Hom}_{\gerbe{G}(V)} \bigl( \mrm{rest}^{\gerbe{G}}_{V/U}(i),
\mrm{rest}^{\gerbe{G}}_{V/U}(j) \bigr) . \]
For an inclusion $V_1 \subset V_0 \subset U$ of open sets, the
restriction function 
\[ \mrm{rest}^{\gerbe{G}}(i,j)_{V_1 / V_0} :
\gerbe{G}(i,j)(V_0) \to \gerbe{G}(i,j)(V_1) \]
is the composed function
\[ \begin{aligned}
& \opn{Hom}_{\gerbe{G}(V_0)} \bigl(
\mrm{rest}^{\gerbe{G}}_{V_0/U}(i),
\mrm{rest}^{\gerbe{G}}_{V_0/U}(j) \bigr) \\
& \quad \xar{\mrm{rest}^{\gerbe{G}}_{V_1 / V_0}}
\opn{Hom}_{\gerbe{G}(V_0)} \bigl(
(\mrm{rest}^{\gerbe{G}}_{V_1/V_0} \circ
\mrm{rest}^{\gerbe{G}}_{V_0/U}) (i), 
(\mrm{rest}^{\gerbe{G}}_{V_1/V_0} \circ
\mrm{rest}^{\gerbe{G}}_{V_0/U}) (j) \bigr) \\
& \quad \xar{\gamma^{\gerbe{G}}_{V_1 / V_0 / U}}
\opn{Hom}_{\gerbe{G}(V_1)} \bigl(
\mrm{rest}^{\gerbe{G}}_{V_1/U}(i),
\mrm{rest}^{\gerbe{G}}_{V_1/U}(j) \bigr) . 
\end{aligned} \]
Condition (\ref{eqn:6}) ensures that
\[ \mrm{rest}^{\gerbe{G}}(i,j)_{V_2 / V_1} \circ
\mrm{rest}^{\gerbe{G}}(i,j)_{V_1 / V_0} =
\mrm{rest}^{\gerbe{G}}(i,j)_{V_2 / V_0} \]
for an inclusion 
$V_2 \subset V_1 \subset V_0 \subset U$.
Note that the set of sections of this presheaf is
\[ \Gamma(V, \gerbe{G}(i,j)) = \gerbe{G}(V)(i,j) . \]

{} From now on we shall usually write
$i|_V$ instead of $\mrm{rest}^{\gerbe{G}}_{V / U}(i)$, 
for a local object $i \in \opn{Ob} \gerbe{G}(U)$;
and $g|_{V_1}$ instead 
$\mrm{rest}^{\gerbe{G}}(i,j)_{V_1 / V_0}(g)$, for a local morphism 
$g \in \gerbe{G}(i,j)(V_0)$. Furthermore, 
we usually omit reference to the restriction
functors $\mrm{rest}^{\gerbe{G}}_{-/-}$ altogether.

Another convention that we shall adopt from here on is that we
denote the composition in the local categories
$\gerbe{G}(U)$ of a prestack $\gerbe{G}$ by ``$\circ$'', and not by 
``$*$'' as we did up to here.

Let $F : \gerbe{G} \to \gerbe{H}$ be a morphism of prestacks. 
One says that $F$ is a {\em weak equivalence} if it satisfies these conditions:
\begin{enumerate}
\rmitem{i} $F$ is {\em locally essentially surjective on objects}. This mean
that for any open set $U \subset X$,
object $j \in \opn{Ob} \gerbe{H}(U)$ and point $x \in U$, there is an open
set $V$ with $x \in V \subset U$, an object 
$i \in \opn{Ob} \gerbe{G}(V)$, and an isomorphism
$h : F(i) \iso j$ in $\gerbe{H}(V)$. 
\rmitem{ii} For any open set $U$ and $i, j \in \opn{Ob} \gerbe{G}(U)$ the
function
\[ F : \gerbe{G}(U)(i, j) \to \gerbe{H}(U) \bigl( F(i), F(j) \bigr) \]
is bijective. In other words, the functor
$F : \gerbe{G}(U) \to \gerbe{H}(U)$
is fully faithful. 
\end{enumerate}

A prestack $\gerbe{G}$ is called a {\em stack} if it satisfies 
these two conditions:
\begin{enumerate}
\rmitem{a} {\em Descent for morphisms}. This means that the presheaves of
morphisms $\gerbe{G}(i,j)$ are all sheaves. 
\rmitem{b} {\em Descent for objects}. This means 
that given an open set $U$, an open covering 
$U = \bigcup_{k \in K} U_k$, objects
$i_k \in \opn{Ob} \gerbe{G}(U_k)$, and isomorphisms
\[ g_{k_0, k_1} \in \gerbe{G}(U_{k_0, k_1})(i_{k_0}|_{U_{k_0, k_1}}, 
i_{k_1}|_{U_{k_0, k_1}}) \] 
that satisfy
\[ g_{k_1, k_2}|_{U_{k_0, k_1, k_2}} \circ 
g_{k_0, k_1}|_{U_{k_0, k_1, k_2}}  = g_{k_0, k_2}|_{U_{k_0, k_1, k_2}}  , \]
there exists an object
$i \in \gerbe{G}(U)$, 
and isomorphisms 
$g_k \in \gerbe{G}(U_{k})(i|_{U_{k}}, i_{k})$,
such that
\[ g_{k_0, k_1} \circ g_{k_0}|_{U_{k_0, k_1}}  = g_{k_1}|_{U_{k_0, k_1}} . \]
\end{enumerate}

Observe that by condition (a), the object $i \in \gerbe{G}(U)$ in condition (b)
is unique up to a unique isomorphism. A prestack $\gerbe{G}$ satisfying
condition (a) is sometimes called a {\em separated prestack}. 

We denote by $\bsym{\cat{Stack}}\, X$ the full sub $2$-category of
$\bsym{\cat{PreStack}}\, X$ gotten by taking all stacks, all
$1$-morphisms between them, and all $2$-morphisms between these
$1$-morphisms. 

It is not hard to see that a morphism of stacks $F : \gerbe{G} \to \gerbe{H}$ is
an equivalence if and only if it is a weak equivalence.

There is a {\em stackification} operation, which is analogous to sheafification:
to any prestack $\gerbe{G}$ one assigns a stack
$\til{\gerbe{G}}$, with a morphism of prestacks
$F : \gerbe{G} \to \til{\gerbe{G}}$.
These have the following universal property: given any stack 
$\gerbe{H}$ and morphism 
$E : \gerbe{G} \to \gerbe{H}$,
there is a morphism $\til{E} : \til{\gerbe{G}} \to \gerbe{H}$,
unique up to $2$-isomorphism, such that 
$E \twocong \til{E} \circ F$. 

Recall that a {\em groupoid} is a
category $\cat{G}$ in which all morphisms are isomorphisms. For an object $i$
the set $\cat{G}(i,i)$ is then a group. If the set 
$\cat{G}(i, j) \neq \emptyset$,
then it is a $\cat{G}(j, j)$-$\cat{G}(i, i)$-bitorsor. For 
$g \in \cat{G}(i, j)$ we denote by $\opn{Ad}(g)$ the group isomorphism 
$\cat{G}(i, i) \to \cat{G}(j, j)$ given by
$\opn{Ad}(g)(h) := g \circ h \circ g^{-1}$. 

By a {\em prestack of groupoids} on $X$ we mean a prestack 
$\gerbe{G}$ such that each of the categories $\gerbe{G}(U)$ is
a groupoid.
We denote by $\bsym{\cat{PreStGr}}\, X$
the full sub $2$-category of
$\bsym{\cat{PreStack}}\, X$ gotten by taking all prestacks of groupoids, all
$1$-morphisms between them, and all $2$-morphisms between these
$1$-morphisms. 
If $\gerbe{G}$ is a prestack of groupoids, then the associated stack
$\til{\gerbe{G}}$ is a stack of groupoids. 

We shall be interested in {\em gerbes}. 
A gerbe is a stack of groupoids $\gerbe{G}$ on $X$ that has these two
properties:
\begin{enumerate}
\rmitem{$\dag$} $\gerbe{G}$ is {\em locally nonempty}. What this means is
that any point $x \in X$ has an open neighborhood $U$ such that 
$\opn{Ob} \gerbe{G}(U) \neq \emptyset$. 
\rmitem{$\dag \dag$} $\gerbe{G}$ is {\em locally connected}. This
says that for any $i, j \in \opn{Ob} \gerbe{G}(U)$
and any $x \in X$, there is an open set $V$ such that $x \in V \subset U$ and
$\gerbe{G}(V)(i, j) \neq \emptyset$. 
\end{enumerate}

A gerbe $\gerbe{G}$ is called {\em trivial} if 
$\opn{Ob} \gerbe{G}(X) \neq \emptyset$.

Let $\mcal{G}$ be a sheaf of groups on $X$. By a {\em left
$\mcal{G}$-torsor} on
$X$ we mean a sheaf of sets $\mcal{S}$, with a left $\mcal{G}$-action, such
that $\mcal{S}$ is locally nonempty (i.e.\ each point $x \in X$ has an open
neighborhood $U$ such that $\mcal{S}(U) \neq \emptyset$), 
and for any $s \in \mcal{S}(U)$ the morphism of sheaves of sets
$\mcal{G}|_U \to \mcal{S}|_U$, $g \mapsto g \cdot s$, is an isomorphism. 
The torsor $\mcal{S}$ is trivial if $\mcal{S}(X) \neq \emptyset$. 

Suppose $\gerbe{G}$ is a gerbe on $X$. Given an open
set $U \subset X$ and $i \in \opn{Ob} \gerbe{G}(U)$, there is a sheaf of
groups $\gerbe{G}(i,i)$ on $U$. If $j \in \opn{Ob} \gerbe{G}(U)$
is some other object, then the sheaf of sets 
$\gerbe{G}(i,j)$ is a 
$\gerbe{G}(j,j)$-$\gerbe{G}(i,i)$-bitorsor. Namely, forgetting the
left action by $\gerbe{G}(j,j)$, the sheaf $\gerbe{G}(i,j)$ is a
right $\gerbe{G}(i,i)$-torsor; and vice versa.

We denote by $\bsym{\cat{Gerbe}}\, X$ the full sub-$2$-category of
$\bsym{\cat{PreStGr}}\, X$ gotten by taking all gerbes, all
$1$-morphisms between gerbes, and all $2$-morphisms between these $1$-morphisms.

Here are two prototypical examples of gerbes.

\begin{exa} \label{exa:15}
Let $\cat{I}$ be the groupoid with one object, say $0$, and with
$\cat{I}(0, 0) := \{ \bsym{1}_0 \}$, the trivial group.  
This groupoid is a terminal object in $\twocat{Cat}$, since any category
$\cat{C}$ admits exactly one functor $\cat{C} \to \cat{I}$. 

Now take a topological space $X$, and define a prestack $\gerbe{I}$ on it by
letting $\gerbe{I}(U) := \cat{I}$ for any open set $U$. 
Then $\gerbe{I}$ is a gerbe. The gerbe $\gerbe{I}$ is a terminal
object in $\bsym{\cat{PreStack}}\, X$. Indeed, given any prestack 
$\gerbe{G}$ on $X$ there is a unique morphism of prestacks
$\gerbe{G} \to \gerbe{I}$. We call $\gerbe{I}$ the {\em terminal gerbe}
(because the word ``trivial'' is over-used in this area).  
\end{exa}

\begin{exa} \label{exa:16}
Let $\mcal{G}$ be a sheaf of groups on $X$. For an open set $U$ let
$\cat{Tors}(\mcal{G}|_U)$ be the set of all left $\mcal{G}|_U$-torsors.
This is a groupoid. Given $V \subset U$ there is a functor
\[ \cat{Tors}(\mcal{G}|_U) \to \cat{Tors}(\mcal{G}|_V) , \]
namely $\mcal{S} \mapsto \mcal{S}|_V$. Thus we obtain a prestack of groupoids
$\twocat{Tors} \mcal{G}$ with
\[ (\twocat{Tors} \mcal{G})(U) := \cat{Tors}(\mcal{G}|_U) . \]
Since torsors are locally trivial it follows that 
$\twocat{Tors} \mcal{G}$ is a gerbe, called the {\em gerbe of
$\mcal{G}$-torsors}.
\end{exa}

\begin{rem}
A prestack of groupoids $\gerbe{G}$ is sometimes called a {\em
category fibered in groupoids} over $\cat{Open} X$. 
More precisely, given $\gerbe{G}$, we can construct a category
$\cat{G}$, together with a functor $\Phi : \cat{G} \to \cat{Open} X$
called the fiber functor. The set of objects of $\cat{G}$ is
\[ \opn{Ob} \cat{G} := 
\coprod_{U \in \cat{Open} X} \opn{Ob} \gerbe{G}(U) . \]
For objects $i \in \opn{Ob} \gerbe{G}(U)$ and
$j \in \opn{Ob} \gerbe{G}(V)$ one defines 
\[ \opn{Hom}_{\cat{G}}(i, j) := 
\opn{Hom}_{\gerbe{G}(U)}(i, j|_U) \]
if $U \subset V$; and 
$\opn{Hom}_{\cat{G}}(i, j) := \emptyset$ otherwise. 
The fiber functor $\Phi : \cat{G} \to \cat{Open} X$ is
$\Phi(i) := U$ for $i \in \opn{Ob} \gerbe{G}(U)$, and
$\Phi(g) := (U \to V)$ for 
$g \in \opn{Hom}_{\cat{G}}(i, j)$ as above.

Conversely, the prestack $\gerbe{G}$ can be recovered from the data
$\Phi : \cat{G} \to \cat{Open} X$.

For stacks of groupoids arising from moduli problems it is often more
natural to use the fibered category approach (cf.\ \cite{LMB}); but
for our applications in \cite{Ye}, the pseudofunctor approach is more
suitable.
\end{rem}

%\cleardoublepage
\section{Extensions of Gerbes}
\label{sec.ext}

We begin by taking certain basic notions about groups (such as normal subgroup
and center) and generalizing them to groupoids. 
As a matter of convenience we often refer to a functor 
$F : \cat{G} \to \cat{H}$ 
between groupoids as a morphism. (Indeed this is a $1$-morphism in the 
$2$-category $\twocat{Groupoid}$ of groupoids.)

\begin{dfn} \label{dfn:3}
Let $\cat{G}$ be a groupoid. A {\em normal subgroupoid} of $\cat{G}$ is a
subgroupoid $\cat{N}$ satisfying the following three conditions. 
\begin{enumerate}
\rmitem{i} $\opn{Ob} \cat{N} = \opn{Ob} \cat{G}$.
\rmitem{ii} For every $i, j \in \opn{Ob} \cat{G}$ and 
$g \in \cat{G}(i, j)$ there is equality
\[ \opn{Ad}(g) \bigl( \cat{N}(i, i) \bigr)  = \cat{N}(j, j) . \]
\rmitem{iii} $\cat{N}$ is totally disconnected, i.e.\ 
$\cat{N}(i, j) = \emptyset$ for $i \neq j$.
\end{enumerate}
\end{dfn}

In particular for every $i \in \opn{Ob} \cat{G}$ the
group $\cat{N}(i, i)$ is normal subgroup of the automorphism group 
$\cat{G}(i, i)$. 

Note that a normal subgroupoid $\cat{N}$ is the same as a collection 
$\{ N_i \}_{i \in \opn{Ob} \cat{G}}$ 
of subgroups $N_i \subset \cat{G}(i, i)$, satisfying the obvious variant of
condition (ii). 

The {\em trivial normal subgroupoid} of $\cat{G}$ is the normal subgroupoid
$\cat{N}$ for which all the groups $\cat{N}(i, i)$ are trivial; namely 
$\cat{N}(i, i) = \{ \bsym{1}_i \}$.

Let  $F : \cat{G} \to \cat{H}$ be a morphism of groupoids. 
For $i \in \opn{Ob} \cat{G}$ let
\[ \opn{Ker}(F)(i, i) := \opn{Ker} \bigl( F : \cat{G}(i,i) \to 
\cat{H} ( F(i), F(i) ) \bigr) . \]
This is a normal subgroup of $\cat{G}(i, i)$. 
Moreover, an easy calculation shows that the collection of subgroups 
$\{ \opn{Ker}(F)(i, i) \}_{i \in \opn{Ob} \cat{G}}$ 
is a normal subgroupoid of $\cat{G}$, which we denote by 
$\opn{Ker}(F)$. 

\begin{dfn}
Let $F : \cat{G} \to \cat{H}$ be a morphism of groupoids. We say that 
$F$ is a {\em weak epimorphism} if it satisfies these conditions:
\begin{enumerate}
\rmitem{i} $F$ is essentially surjective objects. Namely for any
$j \in \opn{Ob} \cat{H}$ there exists some $i \in \opn{Ob} \cat{G}$ such that
$\cat{H}(F(i), j) \neq \emptyset$.
\rmitem{ii} $F$ is surjective on sets of morphisms. This means that for any
$i, j \in \opn{Ob} \cat{G}$ the function
\[ F : \cat{G}(i, j) \to \cat{H} \bigl( F(i), F(j) \bigr) \]
is surjective.
\end{enumerate}
\end{dfn}

Observe that if $F : \cat{G} \to \cat{H}$ is a weak epimorphism whose kernel 
$\opn{Ker}(F)$ is the trivial normal subgroupoid of $\cat{G}$, then $F$ is an
equivalence.

\begin{dfn}
By an {\em extension of groupoids} we mean a diagram of morphisms of groupoids
\[ \cat{N} \xar{E} \cat{G} \xar{F} \cat{H} , \]
such that $F$ is a weak epimorphism,  
$\cat{N} = \opn{Ker}(F)$, and $E : \cat{N} \to \cat{G}$ is the inclusion.
\end{dfn}

By analogy with the case of groups we often write
\begin{equation} \label{eqn:20}
1 \to \cat{N} \to \cat{G} \xar{F} \cat{H} \to 1  
\end{equation}
for an extension of groupoids. But this is only a suggestive notation -- we do
not view the symbols ``$1$'' as groupoids. (We could, but then the first $1$
has to be replaced by the trivial normal subgroupoid of $\cat{G}$, and
the second $1$ by the terminal groupoid of Example \ref{exa:15}.)

Extensions of groupoids behave very much like extensions of groups. 
Suppose $\cat{G}$ is a groupoid and $\cat{N} \subset \cat{G}$ is a normal
subgroupoid. Then there is an extension of groupoids (\ref{eqn:20}). 
The groupoid $\cat{H}$ in this extension is unique up to equivalence.
One could choose $\cat{H}$ such that the function
$F : \opn{Ob} \cat{G} \to \opn{Ob} \cat{H}$ is bijective; and that would make 
$\cat{H}$ unique up to isomorphism.

Next suppose 
$F : \cat{G} \to \cat{H}$,
$F' : \cat{G}' \to \cat{H}'$ and 
$D : \cat{G} \to \cat{G}'$ are morphisms of groupoids, such that 
$D \bigl( \opn{Ker}(F) \bigr) \subset \opn{Ker}(F')$. 
Then there is a morphism of groupoids
$E : \cat{H} \to \cat{H}'$, unique up to $2$-isomorphism, such that the
diagram
\[ \UseTips \xymatrix @C=5ex @R=5ex {
\cat{G}
\ar[r]^{F}
\ar[d]_{D}
& 
\cat{H}
\ar[d]^{E}
\\
\cat{G}'
\ar[r]_{F'}
& 
\cat{H}'
}  \]
commutes up to $2$-isomorphism. 

Now consider a groupoid $\cat{G}$. 
For any $i \in \opn{Ob} \cat{G}$ we have the center
$\opn{Z}(\cat{G}(i, i))$ of the automorphism group
$\cat{G}(i, i)$. 
Given any pair of objects
$i, j \in \opn{Ob} \cat{G}$, and any isomorphism
$g \in \cat{G}(i, j)$, we have
\[ \opn{Ad}(g) \bigl( \opn{Z}(\cat{G}(i, i)) \bigr)  = 
\opn{Z}(\cat{G}(j, j)) . \]
Therefore the collection of subgroups 
$\{ \opn{Z}(\cat{G}(i, i)) \}_{i \in \opn{Ob} \cat{G}}$ 
is a normal subgroupoid of $\cat{G}$, 
which we denote by $\opn{Z}(\cat{G})$, and call the {\em center} of $\cat{G}$.

\begin{dfn}
\begin{enumerate}
\item  Let $\cat{G}$ be a groupoid. A {\em central subgroupoid} of
$\cat{G}$ is any normal subgroupoid $\cat{N}$ that is contained in
$\opn{Z}(\cat{G})$.
\item A {\em central extension of groupoids} is an extension of groupoids
(\ref{eqn:20}) such that $\cat{N}$ is a central subgroupoid of $\cat{G}$.
\end{enumerate}
\end{dfn}

Suppose $\cat{G}$ is a nonempty and connected groupoid, and 
$\cat{N}$ is a central subgroupoid of $\cat{G}$. 
Take any $i, j \in \opn{Ob} \cat{G}$ and
$g, g' \in \cat{G}(i, j)$. Then the group isomorphisms
\[ \opn{Ad}(g) , \opn{Ad}(g') : \cat{N}(i, i) \to \cat{N}(j, j) \]
are equal. In this way we can canonically identify the abelian groups 
$\cat{N}(i, i)$, for $i \in \opn{Ob} \cat{G}$, and view them as a single abelian
group.

When we are given a central extension of groupoids (\ref{eqn:20})
in which $\cat{G}$ is nonempty and connected, we can replace the 
central subgroupoid $\cat{N}$ by a single abelian group $N$ as explained above,
and the extension becomes
\[ 1 \to N \to \cat{G} \xar{F} \cat{H} \to 1 . \]

So far for the discrete situation; now we geometrize. 
Let $X$ be a topological space. 
Suppose $\gerbe{G}$ is a gerbe on $X$. 
By a local object $i$ of $\gerbe{G}$ we mean an object
$i \in \opn{Ob} \gerbe{G}(U)$ for some open set $U \subset X$. 
If $i, j$ are two local objects, defined on open sets $U, V$ respectively, then
by $\gerbe{G}(i,j)$ we mean the corresponding sheaf of isomorphisms on 
$U \cap V$. By a local isomorphism $g : i \iso j$
we mean an isomorphism 
$g \in \gerbe{G}(i,j)(W)$ for some open set 
$W \subset U \cap V$. 
Such $g$ gives rise to an isomorphism of sheaves of groups
\[ \opn{Ad}(g) : \gerbe{G}(i,i)|_W \iso \gerbe{G}(j,j)|_W . \]

\begin{dfn} \label{dfn.9}
Let $\gerbe{G}$ be a gerbe on $X$. A {\em normal subprestack of groupoids}
of $\gerbe{G}$ is a subprestack $\gerbe{N}$ of $\gerbe{G}$
with these two properties:
\begin{itemize}
\rmitem{i} For every open set $U$ the category 
$\gerbe{N}(U)$ is a normal subgroupoid of $\gerbe{G}(U)$
(see Definition \ref{dfn:3}). In particular
$\opn{Ob} \gerbe{N}(U) = \opn{Ob} \gerbe{G}(U)$, and
$\gerbe{N}(U)$ is totally disconnected. 
\rmitem{ii} For every local object $i$ of $\gerbe{N}$ the presheaf
$\gerbe{N}(i, i)$ is a sheaf. 
\end{itemize}

Since the full name is too long, we simply call such $\gerbe{N}$ a 
 {\em normal subgroupoid} of $\gerbe{G}$.
\end{dfn}

Here is what the definition amounts to. For every local object $i$ of 
$\gerbe{G}$ there is a subsheaf of groups 
$\gerbe{N}(i, i) \subset \gerbe{G}(i, i)$. The condition is that for any
local objects $i$ and $j$, and any local isomorphism 
$g : i \to j$, one has
\[ \opn{Ad}(g) \bigl( \gerbe{N}(i, i) \bigr) = \gerbe{N}(j, j) . \]

Warning: a normal subgroupoid of a gerbe is usually not a gerbe, nor even a
stack. 

\begin{prop} \label{prop:13}
Given a morphism of gerbes $F : \gerbe{G} \to \gerbe{H}$, there is a unique 
normal subgroupoid $\gerbe{N}$ of $\gerbe{G}$ such that 
\[ \gerbe{N}(U) = 
\opn{Ker} \bigl( F : \gerbe{G}(U) \to \gerbe{H}(U) \bigr) \]
for every open set $U$. 
\end{prop}

\begin{proof}
The formula defines a subprestack of groupoids $\gerbe{N}$ of $\gerbe{G}$.
We know that the groupoid $\gerbe{N}(U)$ is normal in $\gerbe{G}(U)$. 
And for any local object $i$ of $\gerbe{G}$ we have 
\[  \gerbe{N}(i, i) = \opn{Ker} \bigl( F : \gerbe{G}(i,i) \to 
\gerbe{H} \bigl( F(i), F(i) \bigr) \bigr)  \]
as presheaves, so $\gerbe{N}(i, i)$ a sheaf. 
\end{proof}

\begin{dfn}
The normal subgroupoid $\gerbe{N}$ in the proposition above is called the
{\em kernel} of $F$, and it is denoted by $\opn{Ker}(F)$
\end{dfn}

\begin{dfn}
Let $F : \gerbe{G} \to \gerbe{H}$ be a morphism of gerbes. We say that 
$F$ is a {\em weak epimorphism} if it satisfies these conditions:
\begin{enumerate}
\rmitem{i} $F$ is {\em locally essentially surjective on objects}. 
Recall that this says that for any open set $U \subset X$,
object $j \in \opn{Ob} \gerbe{H}(U)$ and point $x \in U$, there is an open
set $V$ with $x \in V \subset U$, an object 
$i \in \opn{Ob} \gerbe{G}(V)$, and an isomorphism
$h : F(i) \iso j$ in $\gerbe{H}(V)$. 
\rmitem{ii} $F$ is {\em surjective on isomorphism sheaves}. This says that for
any $i, j \in \opn{Ob} \gerbe{G}(U)$
the map of sheaves of sets
\[ F : \gerbe{G}(i, j) \to \gerbe{H}
\bigl( F(i), F(j) \bigr) \]
is surjective. 
\end{enumerate}
\end{dfn}

Note that if $F: \gerbe{G} \to \gerbe{H}$ is a weak epimorphism such that
$\opn{Ker}(F)$ is the trivial normal subgroupoid of $\gerbe{G}$, then 
$F$ is a weak equivalence, and hence it is an equivalence.

\begin{dfn} \label{dfn:1} 
An {\em extension of gerbes} is a diagram
\[  \gerbe{N} \xar{E} \gerbe{G} \xar{F} \gerbe{H}  \]
of morphisms in $\twocat{PreStGr} X$, such that $\gerbe{G}$ and $\gerbe{H}$
are gerbes, $F$ is a weak epimorphism,  
$\gerbe{N} = \opn{Ker}(F)$, and $E : \gerbe{N} \to \gerbe{G}$ is the inclusion.
\end{dfn}

We often use the notation of ``exact sequence''
\begin{equation} \label{eqn:21}
1 \to  \gerbe{N} \to \gerbe{G} \xar{F} \gerbe{H} \to 1  
\end{equation}
for an extension of gerbes. 

\begin{dfn} \label{dfn.5}
A {\em morphism of extensions of gerbes} is a diagram
\[ \UseTips \xymatrix @C=5ex @R=5ex {
\gerbe{N}
\ar[r]
\ar[d]
&
\gerbe{G}
\ar[r]^{F}
\ar[d]_{D}
& 
\gerbe{H}
\ar[d]_{E}
\\
\gerbe{N}'
\ar[r]
&
\gerbe{G}'
\ar[r]_{F'}
& 
\gerbe{H}'
}  \]
of morphisms in $\twocat{PreStGr} X$, 
where the rows are extensions of gerbes,
the  square on the right is commutative up to $2$-isomorphism,
and the square on the left is commutative.
We denote this morphism of extensions by $(D, E)$. 
\end{dfn}

\begin{thm} \label{thm:1} 
Let $\gerbe{G}$ be a gerbe on $X$, and let
$\gerbe{N}$ be a normal subgroupoid of $\gerbe{G}$. Then there exists
a gerbe $\gerbe{G} / \gerbe{N}$, 
and a morphism of gerbes
$F : \gerbe{G} \to \gerbe{G} / \gerbe{N}$,
with the following properties:
\begin{enumerate}
\rmitem{i} The diagram
\[ 1 \to  \gerbe{N} \to \gerbe{G} \xar{F}
\gerbe{G} / \gerbe{N} \to 1 \]
is an extension of gerbes.
\rmitem{ii} Suppose 
\[ 1 \to  \gerbe{N}' \to \gerbe{G}' \xar{F'}
\gerbe{H}' \to 1  \]
is an extension of gerbes, and 
$D : \gerbe{G} \to \gerbe{G}'$
is a morphism of gerbes, such that 
$D(\gerbe{N}) \subset \gerbe{N}'$. Then there is a 
morphism gerbes 
$E : \gerbe{G} / \gerbe{N} \to \gerbe{H}'$,
unique up to $2$-isomorphism, such that the diagram
\[ \UseTips \xymatrix @C=5ex @R=5ex {
1 
\ar[r]
&
\gerbe{N}
\ar[r]
\ar[d]
&
\gerbe{G}
\ar[r]^{F}
\ar[d]_{D}
& 
\gerbe{G} / \gerbe{N} 
\ar[r]
\ar[d]_{E}
& 
1
\\
1 
\ar[r]
&
\gerbe{N}'
\ar[r]
&
\gerbe{G}'
\ar[r]_{F'}
& 
\gerbe{H}'
\ar[r]
& 
1
}  \]
is a morphism of extensions.
\rmitem{iii} In the situation of property \tup{(ii)}, assume the morphism 
$D : \gerbe{G} \to \gerbe{G}'$ is an equivalence, and the sheaf
homomorphisms
\[ D : \gerbe{N}(i, i) \to \gerbe{N}' \bigl( F(i), F(i) \bigr) \]
are isomorphisms for all local objects $i$ of $\gerbe{G}$. Then $E$ is also an
equivalence.
\end{enumerate}
\end{thm}

Before giving the proof we need some preliminary work. 
Let $U \subset X$ be an open set, and let
$i, j \in \opn{Ob} \gerbe{G}(U)$. The sheaf of sets 
$\gerbe{G}(i, j)$ is a right $\gerbe{G}(i, i)$-torsor on $U$, and
hence it has a right action by the sheaf of groups
$\gerbe{N}(i, i)$. Let
$\bar{\gerbe{G}}(i, j)$ be the sheaf of sets on $U$ associated to the
presheaf
\[ V \mapsto \gerbe{G}(V)(i, j) \, / \, \gerbe{N}(V)(i, i) . \]
There is a surjective sheaf morphism
$\gerbe{G}(i, j) \to \bar{\gerbe{G}}(i, j)$.
If $i = j$ we get a sheaf of groups
$\bar{\gerbe{G}}(i, i)$.

\begin{lem} \label{lem.3} 
There is a unique structure of 
$\bar{\gerbe{G}}(j, j)$-$\bar{\gerbe{G}}(i, i)$-bitorsor
on $\bar{\gerbe{G}}(i, j)$, such that the surjection 
$\gerbe{G}(i, j) \to \bar{\gerbe{G}}(i, j)$ is 
$\gerbe{G}(j, j) \times \gerbe{G}(i, i)$ -equivariant.
\end{lem}

\begin{proof}
Uniqueness is clear. For existence, we have to exhibit a suitable action of the
sheaf of groups 
$\bar{\gerbe{G}}(j, j) \times \bar{\gerbe{G}}(i, i)$
on the sheaf of sets $\bar{\gerbe{G}}(i, j)$. Because of uniqueness, this
is a local question. 

Choose an open set $V \subset X$ that trivializes the bitorsor 
$\gerbe{G}(i, j)$; namely there is some
$g \in \gerbe{G}(i, j)(V)$. 
Then the left action of
$\gerbe{G}(j, j)|_{V}$ on $\gerbe{G}(i, j)|_{V}$
coincides with the right action of $\gerbe{G}(i, i)|_{V}$, via 
the isomorphism of sheaves of groups
\[ \opn{Ad}(g) : \gerbe{G}(i, i)|_{V} \iso 
\gerbe{G}(j, j)|_{V} . \]
Also we have a torsor isomorphism
\[ \gerbe{G}(i, i)|_{V} \iso \gerbe{G}(i, j)|_{V}, \
f \mapsto g \circ f . \]

Let $\bar{g} \in \bar{\gerbe{G}}(i, j)(V)$ be the image of $g$. 
We then have an isomorphism of sheaves of right
$\bar{\gerbe{G}}(i, i)|_V$-sets 
\[ \bar{\gerbe{G}}(i, i)|_{V} \iso \bar{\gerbe{G}}(i, j)|_{V}, \
\bar{f} \mapsto \bar{g} \circ \bar{f} . \]
It follows that $\bar{\gerbe{G}}(i, j)|_{V}$ is a right 
$\bar{\gerbe{G}}(i, i)|_V$-torsor.
On the other hand, the isomorphism $\opn{Ad}(g)$ induces an isomorphism
of sheaves of groups
\[ \phi : \bar{\gerbe{G}}(i, i)_V \iso
\bar{\gerbe{G}}(j, j)_V . \]
We conclude that 
$\bar{\gerbe{G}}(i, j)|_{V}$ is a 
$\bar{\gerbe{G}}(j, j)|_V$-$\bar{\gerbe{G}}(i, i)|_V$-bitorsor.
And for this bitorsor structure, the isomorphism of sheaves of groups is
$\phi = \opn{Ad}(\bar{g})$. An easy calculation shows that the surjection
$\gerbe{G}(i, j)|_V \to \bar{\gerbe{G}}(i, j)|_V$ is 
$\gerbe{G}(j, j)|_V \times \gerbe{G}(i, i)|_V$ -equivariant.
\end{proof}

\begin{proof}[Proof of the theorem]
The proof is divided into several steps. 

\medskip \noindent
(a) Define a prestack of groupoids $\bar{\gerbe{G}}$,
and a morphism $\gerbe{G} \to \bar{\gerbe{G}}$,
as follows.
For an open set $U \subset X$ the object set is
$\opn{Ob} \bar{\gerbe{G}}(U) := \opn{Ob} \gerbe{G}(U)$.
For a pair of objects $i, j \in \opn{Ob} \bar{\gerbe{G}}(U)$
let $\bar{\gerbe{G}}(i, j)$ be the sheaf of sets from Lemma \ref{lem.3},
and define 
$\bar{\gerbe{G}}(U)(i, j) := \Gamma(U, \bar{\gerbe{G}}(i, j))$.
Next let $\gerbe{G} / \gerbe{N}$ be the stack associated to 
$\bar{\gerbe{G}}$. So $\gerbe{G} / \gerbe{N}$ is a gerbe,
and there is a weak equivalence of prestacks
$\bar{\gerbe{G}} \to \gerbe{G} / \gerbe{N}$.
It is important to note that even though 
$\gerbe{G} / \gerbe{N}$ may have more local objects than
$\bar{\gerbe{G}}$, the isomorphism sheaves (for local objects of 
$\bar{\gerbe{G}}$) are unchanged. 

\medskip \noindent
(b) The morphism of gerbes 
$F : \gerbe{G} \to \gerbe{G} / \gerbe{N}$
we get from step (a) is a weak epimorphism, and its kernel in
$\gerbe{N}$. This proves property (i). 

\medskip \noindent
(c) In this step we prove the existence part of property (ii). 
Let us define a morphism of prestacks
$\bar{D} : \bar{\gerbe{G}} \to \gerbe{H}'$ as follows.
On objects $\bar{D}$ is just $F' \circ D$. 
And on isomorphisms, for local objects $i, j$ of 
$\bar{\gerbe{G}}$, we define
\[ \bar{D} : \bar{\gerbe{G}}(i, j) \to 
\gerbe{H}' \bigl( D(i), D(j) \bigr) \]
to be the unique $\gerbe{G}(i, i)$-equivariant sheaf morphism making
the diagram
\[ \UseTips \xymatrix @C=8ex @R=5ex {
\gerbe{G}(i, j)
\ar[r]
\ar[d]_{D}
&
\bar{\gerbe{G}}(i, j)
\ar[d]_{\bar{D}}
\\
\gerbe{G}' \bigl( D(i), D(j) \bigr)
\ar[r]^(0.5){F'}
&
\gerbe{H}' \bigl( \bar{D}(i), \bar{D}(j) \bigr)
}  \]
commute. 
Due to the universal property of stackification, $\bar{D}$ induces a
morphism of gerbes 
$E : \gerbe{G} / \gerbe{N} \to \gerbe{H}'$; and then 
$(D, E)$ is a morphism of extensions. 

\medskip \noindent
(d) Now we will prove that the morphism $E$ from step (c) is unique up to
$2$-isomorphism. Suppose 
$E' : \gerbe{G} / \gerbe{N} \to \gerbe{H}'$
is some other morphism such that $(D, E')$ is a morphism of extensions. 
By composing the canonical morphism 
$\bar{\gerbe{G}} \to \gerbe{G} / \gerbe{N}$ with $E'$, we
obtain a morphism
$\bar{D}' : \bar{\gerbe{G}} \to \gerbe{H}'$.
We are going to construct a $2$-isomorphism 
$\bar{\eta} : \bar{D} \twoiso \bar{D}'$.

For a local object 
$i \in \opn{Ob} \gerbe{G}(U) = \opn{Ob} \bar{\gerbe{G}}(U)$
let
$j := \bar{D}(i) \in \opn{Ob} \gerbe{H}'(U)$
and
$j' := \bar{D}'(i) \in \opn{Ob} \gerbe{H}'(U)$.
So $j = (F' \circ D)(i)$ and $j' = (E' \circ F)(i)$.
Take any $2$-isomorphism $\eta : E' \circ F \twoiso F' \circ D$.
Then $\eta$ induces a $2$-isomorphism
$\bar{\eta} : \bar{D} \twoiso \bar{D}'$,
which coincides with $\eta$ on objects of $\bar{\gerbe{G}}$,
and is the reduction of $\eta$ modulo 
$\gerbe{N}$ on isomorphisms in $\bar{\gerbe{G}}$.

Because $\gerbe{G} / \gerbe{N}$ is the stackification of 
$\bar{\gerbe{G}}$, and $E, E'$ are the stackifications of 
$\bar{D}, \bar{D}'$ respectively, $\bar{\eta}$ induces a $2$-isomorphism 
$E \twoiso E'$. 

\medskip \noindent
(e) Finally we shall prove property (iii). 
The morphism $\bar{D} : \bar{\gerbe{G}} \to \gerbe{H}'$
is locally surjective on objects. This is because $\bar{\gerbe{G}}$
and $\gerbe{G}$ have the same local objects; 
$\gerbe{G} \to \gerbe{G}'$ is locally bijective on objects; and
$\gerbe{G}' \to \gerbe{H}'$ is 
locally surjective on objects.

By construction, for any pair of local objects
$i, j$ of $\bar{\gerbe{G}}$ we have
\[ \bar{\gerbe{G}}(i, j) = \gerbe{G}(i, j) / \gerbe{N}(i, i) \]
as sheaves of sets. On the other hand
\[ D : \gerbe{N}(i, i) \to
\gerbe{N}' \bigl( D(i), D(i) \bigr) \]
is an isomorphism of sheaves of groups, and 
\[ D : \gerbe{G}(i, j) \to 
\gerbe{G}' \bigl( D(i), D(j) \bigr) \]
is an isomorphism of torsors. 
Since 
\[ \gerbe{H}' \bigl( E(i), E(j) \bigr) \cong
\gerbe{G}' \bigl( D(i), D(j) \bigr) / \gerbe{N}' \bigl( D(i), D(i) \bigr) \]
we see that $\bar{D} : \bar{\gerbe{G}} \to \gerbe{H}'$
is a weak equivalence. Therefore 
$E : \gerbe{G} / \gerbe{N} \to \gerbe{H}'$
is an equivalence.
\end{proof}

\begin{cor}
Suppose we are given extensions of gerbes
\[ 1 \to  \gerbe{N} \to \gerbe{G} \xar{F} \gerbe{H} \to 1 
\]
and
\[ 1 \to  \gerbe{N}' \to \gerbe{G}' \xar{F'} \gerbe{H}' \to 1
, \]
and a morphism of gerbes
$D : \gerbe{G} \to \gerbe{G}'$,
such that $D(\gerbe{N}) \subset \gerbe{N}'$. Then there is a 
morphism of gerbes $E : \gerbe{H} \to \gerbe{H}'$, 
unique up to $2$-isomorphism, such that the diagram
\[ \UseTips \xymatrix @C=5ex @R=5ex {
1 
\ar[r]
&
\gerbe{N}
\ar[r]
\ar[d]
&
\gerbe{G}
\ar[r]^{F}
\ar[d]_{D}
& 
\gerbe{H} 
\ar[r]
\ar[d]_{E}
& 
1
\\
1 
\ar[r]
&
\gerbe{N}'
\ar[r]
&
\gerbe{G}'
\ar[r]_{F'}
& 
\gerbe{H}'
\ar[r]
& 
1
}  \]
is a morphism of extensions.
\end{cor}

\begin{proof}
By the theorem we can replace $\gerbe{H}$ with the equivalent gerbe
$\gerbe{G} / \gerbe{N}$. Now we can use property (ii) of the
theorem.
\end{proof}

Given a sheaf of groups $\mcal{G}$ on $X$ and an 
open set $U \subset X$, we write
$\mcal{G}(U) := \Gamma(U, \mcal{G})$. The center of this group is denoted by
$\mrm{Z}(\mcal{G}(U))$. Since the center is not functorial, one has to be
careful what we mean by the center of the sheaf $\mcal{G}$. The correct
definition seems to be as follows. 

\begin{dfn}
Let $\mcal{G}$ be a sheaf of groups on $X$. The {\em center} of $\mcal{G}$ is
the sheaf of groups $\mrm{Z}(\mcal{G})$ whose group of sections on an open set
$U$ is
\[ \mrm{Z}(\mcal{G})(U) := 
\{ g \in \mcal{G}(U) \mid g|_V \in \mrm{Z}(\mcal{G}(V))
\text{ for any } V \subset U \} . \]
\end{dfn}

\begin{prop} \label{prop:14}
Given a gerbe $\gerbe{G}$, there is a unique 
normal subgroupoid $\gerbe{N}$ of $\gerbe{G}$ such that 
\[ \gerbe{N}(i, i) = \mrm{Z} \bigl( \gerbe{G}(i, i) \bigr) \]
for every local object $i$ of $\gerbe{G}$.
\end{prop}

The proof is like that of Proposition \ref{prop:13}.

\begin{dfn}
Let $\gerbe{G}$ be a gerbe. 
\begin{enumerate}
\item The normal subgroupoid $\gerbe{N}$ in Proposition \ref{prop:14} is
called the {\em center} of $\gerbe{G}$, and it is denoted by
$\opn{Z}(\gerbe{G})$.
\item A {\em central subgroupoid} of  $\gerbe{G}$ is a 
normal subgroupoid $\gerbe{N}$ of $\gerbe{G}$ that is contained in 
$\opn{Z}(\gerbe{G})$.
\end{enumerate}
\end{dfn}

\begin{prop} \label{prop:12}
Let $\gerbe{N}$ be a central subgroupoid of  $\gerbe{G}$. 
Then there is sheaf of abelian groups $\mcal{N}$, together with an 
isomorphism of sheaves of groups
\begin{equation} \label{eqn:23}
\chi_i : \mcal{N}|_U \to \gerbe{N}(i, i)
\end{equation}
for any open set $U$ and object $i \in \opn{Ob} \gerbe{G}(U)$,
satisfying this condition:
\begin{itemize}
\rmitem{$\diamondsuit$}
Given any $i \in \opn{Ob} \gerbe{G}(U)$,  $j \in \opn{Ob} \gerbe{G}(V)$, 
$W \subset U \cap V$, and $g \in \gerbe{G}(i, j)(W)$, one has
\[ \chi_j \circ \opn{Ad}(g) = \chi_j  \]
as sheaf homomorphisms 
$\mcal{N}|_W \to \mcal{G}(j, j)|_W$. 
\end{itemize}
The sheaf $\mcal{N}$ is unique up to a unique isomorphism.

Conversely, given a sheaf of abelian  groups $\mcal{N}$, together with a
collection of injective sheaf homomorphisms
\[ \chi_i : \mcal{N}|_U \to \opn{Z}(\gerbe{G}(i, i))  \]
for $i \in \opn{Ob} \gerbe{G}(U)$, 
that satisfy condition \tup{($\diamondsuit$)}, there exists a unique central
subgroupoid $\gerbe{N}$ of  $\gerbe{G}$ such that 
\tup{(\ref{eqn:23})} are isomorphisms.
\end{prop}

\begin{proof}
This is due to the local nature of gerbes; cf.\ part (a) of the proof of
Theorem \ref{thm:1}.
\end{proof}

The last proposition says that a central subgroupoid $\gerbe{N}$ of a gerbe
$\gerbe{G}$ can be viewed as a single sheaf of abelian groups.

Finally we can explain the title of the paper. 

\begin{dfn} \label{dfn:2}
A {\em central extension of gerbes} is an extension of gerbes 
\[  1 \to \gerbe{N} \xar{} \gerbe{G} \xar{F} \gerbe{H} \to 1  \]
such that $\gerbe{N}$ is a central subgroupoid of $\gerbe{G}$.
\end{dfn}

Using Proposition \ref{prop:12} to replace $\gerbe{N}$ with a 
sheaf of abelian groups $\mcal{N}$, we can rewrite this central extension as
\begin{equation} \label{eqn:27}
 1 \to \mcal{N} \xar{} \gerbe{G} \xar{F} \gerbe{H} \to 1 .
\end{equation}

Here are a couple of examples of central extensions of gerbes. The first is
somewhat tautological.

\begin{exa}
Suppose $\gerbe{N}$ is a central subgroupoid of a gerbe $\gerbe{G}$.
Then the extension of gerbes 
\[ 1 \to \gerbe{N} \to \gerbe{G} \xar{F} \gerbe{G} / \gerbe{N} \to 1  \]
from Theorem \ref{thm:1} is central.
\end{exa}

The next example was suggested to us by the referee.

\begin{exa} \label{exa:27}
Let 
\[ 1 \to \mcal{N} \to \mcal{G} \xar{F} \mcal{H} \to 1 \]
be a central extension of sheaves of groups on $X$. Given an open set $U$ 
and a $\mcal{G}|_U$-torsor $\mcal{S}$, let us denote by $F(\mcal{S})$ the
induced $\mcal{H}|_U$-torsor. This operation gives rise to a morphism of gerbes
\begin{equation} \label{eqn:24}
F : \twocat{Tors} \gerbe{G} \to \twocat{Tors} \gerbe{H} 
\end{equation}
(cf.\ Example \ref{exa:16}). Since locally any $\mcal{H}$-torsor is
trivial, it is locally induced from a $\mcal{G}$-torsor. This says that 
(\ref{eqn:24}) is locally essentially surjective on objects. 

Now for any $\mcal{G}$-torsor $\mcal{S}$, locally we have a (noncanonical)
isomorphism of sheaves of groups
\[ (\twocat{Tors} \gerbe{G})(\mcal{S}, \mcal{S}) \cong 
\mcal{G}^{\mrm{op}} . \]
Likewise for $\mcal{H}$-torsors. This implies that the 
morphism of gerbes (\ref{eqn:24}) is surjective on isomorphism sheaves, and
its kernel is a central subgroupoid of $\gerbe{G}$, isomorphic to
the sheaf $\mcal{N}$. In this way we get a  central extensions of gerbes
\begin{equation} \label{eqn:25}
1 \to \mcal{N} \to \twocat{Tors} \gerbe{G} \xar{F} \twocat{Tors} \gerbe{H}
\to 1  .
\end{equation}
\end{exa}

%\cleardoublepage
\section{Obstruction Classes}

We fix a topological space $X$. 
Given a sheaf $\mcal{N}$ of abelian groups on $X$, and an open covering
$\bsym{U} = \{ U_k \}_{k \in K}$ of $X$, there are the \v{C}ech
cohomology groups $\check{\mrm{H}}^p(\bsym{U}, \mcal{N})$
for  $p \geq 0$. Passing to the limit over all such open coverings 
we obtain the \v{C}ech cohomology groups $\check{\mrm{H}}^p(X, \mcal{N})$.

{} From here until the end of this section we consider a central
extension of gerbes
\begin{equation} \label{eqn:5}
1 \to  \mcal{N} \to \gerbe{G} \xar{F} \gerbe{H} \to 1 
\end{equation}
on $X$ (see Definition \ref{dfn:2} and Proposition \ref{prop:12}). 

\begin{construction} \label{const2}
Let $i, j \in \opn{Ob} \gerbe{G}(X)$, and let
$h \in \gerbe{H}(X) \bigl( F(i), F(j) \bigr)$. 
Since $F$ is a weak epimorphism, there exists an open covering 
$\bsym{U} = \{ U_k \}_{k \in K}$ of $X$, such that for every $k \in K$ 
there is an isomorphism $g_k \in \gerbe{G}(U_k)(i, j)$ with $F(g_k) = h$.
For every $k_0, k_1 \in K$ we define
\[ g_{k_0, k_1} := g_{k_1}^{-1} \circ g_{k_0}
\in \gerbe{G}(U_{k_0, k_1})(i, i) . \]
Since $F(g_{k_0, k_1}) = 1$ we see that in fact
$g_{k_0, k_1} \in \mcal{N}(U_{k_0, k_1})$.
An easy calculation shows that the collection
\begin{equation} 
c := \{ g_{k_0, k_1} \}_{k_0, k_1 \in K}
\end{equation}
is a \v{C}ech $1$-cocycle for the covering $\bsym{U}$ with values in
the sheaf of groups $\mcal{N}$. 
\end{construction}

\begin{lem}
Let $i, j \in \opn{Ob} \gerbe{G}(X)$ and let
$h \in \gerbe{H}(X) \bigl( F(i), F(j) \bigr)$. 
Suppose that $c$ and $c'$ are $1$-cocycles with values in $\mcal{N}$, for open
coverings $\bsym{U}$  and $\bsym{U}'$, obtained as in Construction
\tup{\ref{const2}}, for the same isomorphism $h$. 
Then their \v{C}ech cohomology classes 
$[c], [c'] \in \check{\mrm{H}}^1(X, \mcal{N})$ are equal. 
\end{lem}

\begin{proof}
Say $\bsym{U}' = \{ U'_k \}_{k \in K'}$.
Suppose that in the course of constructing the cocycle $c'$ we chose,
for every $k \in K'$, an isomorphism 
$g'_k \in \gerbe{G}(U'_k)(i, j)$ with $F(g'_k) = h$.
So
$g'_{k_0, k_1} = {g'}_{k_1}^{-1} \circ g'_{k_0}$
and $c' = \{ g'_{k_0, k_1} \}_{k_0, k_1 \in K'}$. 

Take some open covering
$\bsym{V} = \{ V_l \}_{l \in L}$ of $X$ which refines both 
$\bsym{U}$ and $\bsym{U}'$. Thus there are functions
$\phi : L \to K$ and $\phi' : L \to K'$, such that
$V_l \subset U_{\phi(l)}$ and $V_l \subset U'_{\phi'(l)}$ for all $l \in L$. 
We get cocycles
\[ \phi^*(c) := \{ g_{\phi(l_0), \phi(l_1)} \}_{ l_0, l_1 \in L} \]
and
\[ {\phi'}^*(c') := \{ g'_{\phi'(l_0), \phi'(l_1)} \}_{ l_0, l_1 \in L} . \]
For any $l \in L$ let
\[ f_l := g_{\phi(l)}^{-1} \circ g'_{\phi'(l)}  \in \gerbe{G}(V_l)(i, i) . \]
Now
\[ F(f_l) = F(g_{\phi(l)})^{-1} \circ F(g'_{\phi'(l)}) = 
h^{-1} \circ h = 1 . \]
So $b := \{ f_l \}_{l \in L}$ is a $0$-cochain 
with values in $\mcal{N}$. 

Take any $l_0, l_1 \in L$. Then
\[ \begin{aligned}
& (f_{l_1}^{-1} \circ f_{l_0}) \circ g_{\phi(l_0), \phi(l_1)} \\
& \quad = 
(g_{\phi(l_1)}^{-1} \circ g_{\phi'(l_1)}')^{-1} \circ
(g_{\phi(l_0)}^{-1} \circ g_{\phi'(l_0)}') \circ 
(g_{\phi(l_1)}^{-1} \circ g_{\phi(l_0)}) \\
& \quad = 
g_{\phi'(l_1)}'^{-1} \circ g_{\phi'(l_0)}' =
g'_{\phi'(l_0), \phi'(l_1)} \, . 
\end{aligned} \]
Denoting the \v{C}ech coboundary operator by $\d$, we see that
\[ \d(b) \cdot \phi^*(c) = {\phi'}^*(c') . \]
\end{proof}

In view of this lemma, the following definition makes sense. 

\begin{dfn} \label{dfn.6}
Let $i, j \in \opn{Ob} \gerbe{G}(X)$, and let
$h \in \gerbe{H}(X) \bigl( F(i), F(j) \bigr)$. 
Take any $1$-cocycle $c$ as in Construction \tup{\ref{const2}}.
We define the obstruction class 
\[ \opn{cl}^1_{F}(h) := [c] \in \check{\mrm{H}}^1(X, \mcal{N}) . \]
\end{dfn}

\begin{thm}[Obstruction to lifting isomorphisms] \label{thm.3}
Consider a central extension of gerbes 
\[ 1 \to  \mcal{N} \to \gerbe{G} \xar{F} \gerbe{H} \to 1 . \]
Let $i, j \in \opn{Ob} \gerbe{G}(X)$, and let
$h \in \gerbe{H}(X) \bigl( F(i), F(j) \bigr)$. 
Then there exists an isomorphism 
$g \in \gerbe{G}(X)(i,j)$ satisfying 
$F(g) = h$ if and only if
\[ \opn{cl}_F^1(h) = 1 . \]
\end{thm}

In other words, the class $\opn{cl}_F^1(h)$ is the obstruction to lifting the
isomorphism 
$h : F(i) \to F(j)$
to an isomorphism 
$g : i \to j$.

\begin{proof}
First assume there exists a lifting $g$. 
We construct a cocycle $c$ as follows: for the open covering 
$\bsym{U} = \{ U_k \}_{k \in K}$ we take 
$K := \{ 0 \}$ and $U_0 := X$. We then take $g_0 := g$. 
The resulting cocycle is $c$ is trivial, and hence
$\opn{cl}_F^1(h) = 1$. 

Conversely, assume that $\opn{cl}_F^1(h) = 1$. 
Let $c' = \{ g'_{k_0, k_1} \}_{k_0, k_1 \in K}$ be a $1$-cocycle that represents
$\opn{cl}_F^1(h)$ on some open covering
$\bsym{U}$. Say $g'_k \in \gerbe{G}(U_k)(i, j)$ are the isomorphisms
chosen in the construction of $c'$, so that
$g'_{k_0, k_1} =  g'^{-1}_{k_1} \circ g'_{k_0}$. 

By replacing $\bsym{U}$ with a suitable refinement, we can assume
that $c'$ is a coboundary; i.e.\ there is a $0$-cochain
$b := \{ f_k \}_{k \in K}$ with values in $\mcal{N}$
such that $c' = \d(b)$. Define
\[ g_k := g'_k \circ f^{-1}_k \in  \gerbe{G}(U_k)(i, j) . \]
A calculation shows that $\{ g_k \}_{k \in K}$ is a $0$-cocycle with values in
the sheaf of sets $\gerbe{G}(i, j)$. 
Hence it glues to a global isomorphism
$g \in \gerbe{G}(X)(i, j)$. And by construction we have
$F(g) = h$.
\end{proof}

\begin{cor} \label{cor:1}
Given a central extension of gerbes as above, let 
$i, j \in \opn{Ob} \gerbe{G}(X)$. Then 
$\gerbe{G}(X)(i, j) \neq \emptyset$
if and only if there exists some 
$h \in \gerbe{H}(X) \bigl( F(i), F(j) \bigr)$
such that
$\opn{cl}_F^1(h) = 1$.
\end{cor}

\begin{proof}
This is an immediate consequence of the theorem.
\end{proof}

\begin{rem}
In an earlier version of this paper we claimed that 
$\opn{cl}_F^1(h) = \opn{cl}_F^1(h')$ for any
$h, h' \in \gerbe{H}(X) \bigl( F(i), F(j) \bigr)$. 
However the referee discovered a mistake in our proof.
\end{rem}

\begin{construction} \label{const1}
Let $j \in \opn{Ob} \gerbe{H}(X)$. 
Choose some open covering
$\bsym{U} = \{ U_k \}_{k \in K}$ of $X$. For every $k \in K$ choose,
if possible, an object
$i_k \in \opn{Ob} \gerbe{G}(U_k)$ and an isomorphism
\[ h_k \in \gerbe{H}(U_k) \bigl( F(i_k), j \bigr)  . \]
For every $(k_0, k_1) \in K \times K$ define
\[ h_{k_0, k_1} := h_{k_1}^{-1} \circ h_{k_0}
\in \gerbe{H}(U_{k_0, k_1}) \bigl( F(i_{k_0}), F(i_{k_1}) \bigr)
. \]
Choose, if possible, an isomorphism
\[ g_{k_0, k_1} \in \gerbe{G}(U_{k_0, k_1})(i_{k_0}, i_{k_1}) \]
that lifts $h_{k_0, k_1}$, i.e.\
$F(g_{k_0, k_1}) = h_{k_0, k_1}$. Define
\[ g_{k_0, k_1, k_2} := g_{k_0, k_2}^{-1} \circ g_{k_1, k_2} \circ
g_{k_0, k_1} \in \gerbe{G}(U_{k_0, k_1, k_2})(i_{k_0}, i_{k_0})
. \]
Thus we get a collection of elements
\begin{equation} \label{eqn:2}
c := \{ g_{k_0, k_1, k_2} \}_{k_0, k_1, k_2 \in K} . 
\end{equation}
\end{construction}

\begin{lem} \label{lem:10}
The collection $c$ from this construction is a \v{C}ech $2$-cocycle with values
in $\mcal{N}$.
\end{lem}

\begin{proof}
Since $F(g_{k_0, k_1, k_2}) = 1$ it follows that 
$g_{k_0, k_1, k_2} \in \mcal{N}(U_{k_0, k_1, k_2})$. 

Let us now calculate the value of the coboundary of $c$ in
$\mcal{N}(U_{k_0, k_1, k_2, k_3})$, using the fact that $\mcal{N}$ is
central in $\mcal{G}$:
\[ \begin{aligned}
& g_{k_1, k_2, k_3} \circ g_{k_0, k_2, k_3}^{-1} \circ 
g_{k_0, k_1, k_3} \circ g_{k_0, k_1, k_2}^{-1} \\
& \quad = g_{k_1, k_2, k_3} 
\circ ( g_{k_0, k_3}^{-1} \circ g_{k_2, k_3} \circ 
g_{k_0, k_2} )^{-1} \\
& \qquad \quad  \circ 
( g_{k_0, k_3}^{-1} \circ g_{k_1, k_3} \circ g_{k_0, k_1} ) 
\circ ( g_{k_0, k_2}^{-1} \circ g_{k_1, k_2} \circ 
g_{k_0, k_1} )^{-1} \\
& \quad = g_{k_1, k_2, k_3} \circ 
g_{k_0, k_2}^{-1} \circ g_{k_2, k_3}^{-1} \circ 
g_{k_0, k_3}  \\
& \qquad  \quad \circ g_{k_0, k_3}^{-1} \circ g_{k_1, k_3} \circ g_{k_0, k_1}
\circ g_{k_0, k_1}^{-1} \circ g_{k_1, k_2}^{-1} \circ 
g_{k_0, k_2} \\
& \quad =  g_{k_0, k_2}^{-1} \circ g_{k_2, k_3}^{-1} 
\circ g_{k_1, k_3} \circ g_{k_1, k_2, k_3}
\circ g_{k_1, k_2}^{-1} \circ 
g_{k_0, k_2} \\
& \quad =  g_{k_0, k_2}^{-1} \circ g_{k_2, k_3}^{-1} 
\circ g_{k_1, k_3} \\
& \qquad \quad  \circ (g_{k_1, k_3}^{-1} \circ g_{k_2, k_3} \circ 
g_{k_1, k_2}) \circ g_{k_1, k_2}^{-1} \circ 
g_{k_0, k_2} \\
& \quad =  1 .
\end{aligned} \]
\end{proof}

\begin{lem} \label{lem.4}
Let $j \in \opn{Ob} \gerbe{H}(X)$. 
Suppose that $c$ and $c'$ are $2$-cocycles with values in $\mcal{N}$, for open
coverings $\bsym{U}$  and $\bsym{U}'$, obtained as in Construction
\tup{\ref{const1}}. Then their \v{C}ech cohomology classes 
$[c], [c'] \in \check{\mrm{H}}^2(X, \mcal{N})$ are equal. 
\end{lem}

\begin{proof}
The cocycle $c'$ is constructed using some open covering
$\bsym{U}' = \{ U'_k \}_{k \in K'}$,
objects
$i'_k \in \opn{Ob} \gerbe{G}(U'_k)$
that lift $j$, isomorphisms
$h'_k \in \gerbe{H}(U'_k) \bigl( F(i'_k), j \bigr)$,
and isomorphisms 
$g'_{k_0, k_1} \in \gerbe{G}(U'_{k_0, k_1})(i'_{k_0}, i'_{k_1})$
that lift 
$h'_{k_0, k_1} := {h'}_{k_1}^{-1} \circ h'_{k_0}$. 

The proof proceeds in four steps, labeled (a)-(d). 

\medskip \noindent 
(a) Suppose $\bsym{U}' = \bsym{U}$, 
$i'_k = i_k$ and $h'_k = h_k$, but we choose some other lifting
$g'_{k_0, k_1}$ of $h_{k_0, k_1}$. 
The $2$-cocycle $c' = \{ g'_{k_0, k_1, k_2} \}$ is
\begin{equation} \label{eqn:4}
g'_{k_0, k_1, k_2} := g_{k_0, k_2}'^{-1} \circ g_{k_1, k_2}' \circ
g_{k_0, k_1}'  \in \gerbe{G}(U_{k_0, k_1, k_2})(i_{k_0}, i_{k_0}) .
\end{equation} 
Now there are unique elements
\[ n_{k_0, k_1} \in \mcal{N}(U_{k_0, k_1})(i_{k_0}, i_{k_0}) \]
such that
\[ g'_{k_0, k_1} = n_{k_0, k_1} \circ g_{k_0, k_1} . \]
Consider the \v{C}ech $1$-cochain 
$b := \{ n_{k_0, k_1} \}$ with values in $\mcal{N}$. A little calculation shows
that
\[ g'_{k_0, k_1, k_2} = g_{k_0, k_1, k_2} \circ 
(n_{k_0, k_2}^{-1} \circ  n_{k_0, k_1} \circ n_{k_1, k_2}) ; \]
so that 
$c' = c \cdot \d(b)$. We see that $c$ and $c'$ have
the same cohomology class. 

\medskip \noindent 
(b) Next suppose $\bsym{U}' = \bsym{U}$ and $i'_k = i_k$, but we 
choose other isomorphisms
$h'_k \in \gerbe{H}(U_k)(F(i_k), j)$.
Define
\[ h'_{k_0, k_1} := h_{k_1}'^{-1} \circ h'_{k_0} 
\in \gerbe{H}(U_{k_0, k_1})(F(i_{k_0}), F(i_{k_1})) . \]
Consider the elements
\[ f_k := h_k^{-1} \circ h'_k \in \gerbe{H}(U_k)
(F(i_k), F(i_k)) . \]
Take some open covering $\bsym{V} = \{ V_l \}_{l \in L}$ that refines 
$\bsym{U}$, with comparison function $\phi : L \to K$, such that for every 
$l \in L$ the isomorphism $f_{\phi(l)}$ lifts to some
$g_l \in \bsym{\gerbe{G}}(V_l)(i_{\phi(l)}, i_{\phi(l)})$.
This is possible since $F$ is locally surjective on isomorphism sheaves. 
By replacing $\bsym{U}$ with $\bsym{V}$, we can now assume that each
$f_k$ lifts to some
$g_k \in \bsym{\gerbe{G}}(U_k)(i_{k}, i_{k})$.

Now let us define
\[ g''_{k_0, k_1} := g_{k_1}^{-1} \circ g_{k_0, k_1} \circ
g_{k_0} \in \gerbe{G}(U_{k_0, k_1})(i_{k_0}, i_{k_1}) . \]
Then $g''_{k_0, k_1}$ is a lifting of $h'_{k_0, k_1}$. Proceeding as in equation
(\ref{eqn:4}), we obtain a \v{C}ech $2$-cocycle
$c'' = \{ g''_{k_0, k_1, k_2} \}$.
However, it is easy to see that 
\[ g''_{k_0, k_1, k_2} = g_{k_0}^{-1} \circ g_{k_0, k_1, k_2} \circ
g_{k_0} . \]
Since $g_{k_0, k_1, k_2}$ is central in $\gerbe{G}$ it follows
that in fact
$g''_{k_0, k_1, k_2} =  g_{k_0, k_1, k_2}$, so that $c'' = c$.
On the other hand, from step (a) we see that 
$[c''] = [c']$ in $\check{\mrm{H}}^2(X, \mcal{N})$. 

\medskip \noindent 
(c) Now suppose $\bsym{U}' = \bsym{U}$, but 
we choose another object
$i'_k \in \opn{Ob} \gerbe{G}(U_k)$ for each $k$. 
Take some open covering $\bsym{V} = \{ V_l \}_{l \in L}$ that refines 
$\bsym{U}$, with comparison function $\phi : L \to K$, such that for every 
$l \in L$ one has
\[ \gerbe{G}(V_l)(i_{\phi(l)}, i'_{\phi(l)}) \neq \emptyset . \]
This can be done because $\gerbe{G}$ is locally connected. 
After replacing $\bsym{U}$ with $\bsym{V}$, we can assume that there is some
$f_k \in \gerbe{G}(U_k)(i_k, i'_k)$ for every $k \in K$.

In view of steps (a-b) we might as well take
\[ h'_k := h_k \circ F(f_k)^{-1} \in 
\gerbe{H}(U_k) \bigl( F(i'_k), j \bigr) , \]
and then lift
\begin{equation}
h'_{k_0, k_1} := h_{k_1}'^{-1} \circ h'_{k_0} 
\in \gerbe{H}(U_{k_0, k_1}) \bigl( F(i'_{k_0}), F(i'_{k_1}) \bigr)
\end{equation}
to
\[ g_{k_0, k_1}' := f_{k_1} \circ g_{k_0, k_1} \circ f_{k_0}^{-1} 
\in \gerbe{G}(U_{k_0, k_1})(i'_{k_0}, i'_{k_1}) . \]
The resulting $2$-cocycle $c' = \{ g'_{k_0, k_1, k_2} \}$ defined as in 
(\ref{eqn:4}) will satisfy
\[ g'_{k_0, k_1, k_2} = f_{k_0}^{-1}  \circ 
g_{k_0, k_1, k_2} \circ f_{k_0} . \]
Because $\mcal{N}$ is central we get $c' = c$.

\medskip \noindent 
(d) Finally let's see what happens when we take a new open covering
$\bsym{U}' = \{ U_k \}_{k \in K'}$ of $X$, for which we can construct a
cocycle $c'$. Let $\bsym{V} = \{ V_l \}_{l \in L}$ be a common refinement,
namely there are functions $\phi : L \to K$ and $\phi' : L \to K'$, such that
$V_l \subset U_{\phi(l)}$ and $V_l \subset U'_{\phi'(l)}$ for all $l \in L$. 
Let $\phi^*(c)$ and $\phi'^*(c')$ be the pullback $2$-cocycles on the open
covering $\bsym{V}$. These are both cocycles that are constructed like in
Construction \ref{const1}, for the obvious choices of objects etc. By steps
(a-c) we know that $[\phi^*(c)] = [\phi'^*(c')]$. But on the other hand
$[\phi^*(c)] = [c]$ and $[\phi'^*(c')] = [c']$.
\end{proof}

The lemma justifies the next definition.

\begin{dfn} \label{dfn.4}
Let $j \in \opn{Ob} \gerbe{H}(X)$.
If there exists a $2$-cocycle $c$ as in Construction
\tup{\ref{const1}}, for some open covering $\bsym{U}$, then we
define the {\em obstruction class to lifting objects} to be
\[ \opn{cl}^2_F(j) := [c] \in \check{\mrm{H}}^2(X, \mcal{N}) . \]
Otherwise we say that this obstruction class is undefined.
\end{dfn}

In Section \ref{sec.suff} we shall see sufficient conditions for the
obstruction class $\opn{cl}^2_F(j)$ to be defined.

\begin{prop} \label{prop:10}
Let $j \in \opn{Ob} \gerbe{H}(X)$ be such that the obstruction class
$\opn{cl}^2_F(j)$ is defined. 
Suppose $j' \in \opn{Ob} \gerbe{H}(X)$ is such that
$\gerbe{H}(X)(j,j') \neq \emptyset$. Then
the obstruction class $\opn{cl}_F^2(j')$ is also defined, and moreover
\[ \opn{cl}^2_F(j') = \opn{cl}_F^2(j) . \]
\end{prop}

What the proposition says is that two isomorphic objects have the same 
obstruction class. 

\begin{proof}
We want to construct a \v{C}ech $2$-cocycle $c'$, starting with $j'$ 
instead of $j$. Take any $f \in \gerbe{H}(X)(j, j')$. 
Using this isomorphism we may define
\[ h'_k := f \circ h_k \in \gerbe{H}(U_k)(F(i_k), j') , \]
where $i_k$ is the lifting of $j$ that was used in the construction of $c$,
and $h_k \in \gerbe{H}(U_k)(F(i_k), j)$ is the isomorphism that was
chosen. 

Let 
\[ h'_{k_0, k_1} := h_{k_1}'^{-1} \circ h_{k_0}'
 \in \gerbe{H}(U_{k_0, k_1})(F(i_{k_0}), F(i_{k_1})) . \]
Then $h'_{k_0, k_1} = h_{k_0, k_1}$; so continuing with 
Construction \ref{const1} we get a cocycle $c'$ that equals $c$. 
\end{proof}

\begin{thm}[Obstruction to lifting objects] \label{thm.2}
Consider a central extension of gerbes 
\[ 1 \to  \mcal{N} \to \gerbe{G} \xar{F} \gerbe{H} \to 1 . \]
Let $j \in \opn{Ob} \gerbe{H}(X)$ be such that
the obstruction class $\opn{cl}_F^2(j)$ is defined. Then 
there exists an object 
$i \in \opn{Ob} \gerbe{G}(X)$ with 
\[ \gerbe{H}(X) \bigl( F(i), j \bigr) \neq \emptyset \]
if and only if
\[ \opn{cl}_F^2(j) = 1 . \]
\end{thm}

What the theorem says is that $\opn{cl}_F^2(j)$ is the obstruction to lifting 
$j$ to an object of $\gerbe{G}(X)$.

\begin{proof}
Assume $j$ lifts to an object $i \in \opn{Ob} \gerbe{G}(X)$.
So there exists some isomorphism
$h \in \gerbe{H}(X)(F(i), j)$.
In Construction \ref{const1} we may choose
$i_k := i|_{U_k} \in \opn{Ob} \gerbe{G}(U_k)$. 
Having done so, we take
\[ h_k := h|_{U_k} \in \gerbe{H}(U_k)(F(i_k), j) . \]
Proceeding with the construction, we get
\[ h_{k_0, k_1} = 1 \in 
\gerbe{G}(U_{k_0, k_1})(F(i), F(i)) , \]
which can then be lifted to
\[ g_{k_0, k_1} = 1 \in \gerbe{G}(U_{k_0, k_1})(i, i)  . \]
The resulting $2$-cocycle $c = \{ g_{k_0, k_1, k_2} \}$ is trivial.

Conversely, suppose $\opn{cl}^2_F(j) = 1$. From 
construction \ref{const1} and the choices made there we get a a 
$2$-cocycle $c = \{ g_{k_0, k_1, k_2} \}$ with values in $\mcal{N}$,
on some open covering $\bsym{U}$. By replacing $\bsym{U}$ with a suitable
refinement, we may assume it is a coboundary. Namely there is a $1$-cochain 
$b = \{ f_{k_0, k_1} \}$ with values in $\mcal{N}$, such that 
$c = \d (b)$. 

Consider the isomorphisms
\[ g'_{k_0, k_1} := g_{k_0, k_1} \circ f_{k_0, k_1}^{-1}
\in \gerbe{G}(U_{k_0, k_1})(i_{k_0}, i_{k_1}) , \] 
where $g_{k_0, k_1}$ are the isomorphisms chosen when constructing the cocycle
$c$. 
Then $\{ g_{k_0, k_1}' \}$ is a $1$-cocycle. Since $\gerbe{G}$ is a 
stack, the collection of objects $\{ i_k \}_{k \in K}$ can be glued. I.e.\
there is an object
$i \in \opn{Ob} \gerbe{G}(X)$, and isomorphisms
$g'_k \in \gerbe{G}(U_{k})(i_{k}, i)$,
such that
\[ g_{k_1}'^{-1} \circ g_{k_0}' = g_{k_0, k_1}'  . \]
Define
\[ e_{k} := F(g'_k) \circ h_k^{-1} \in 
\gerbe{H}(U_k) \bigl( j, F(i) \bigr) . \]
Then one checks that 
\[ e_{k_0} = e_{k_1} \in \gerbe{H}(U_{k_0, k_1}) \bigl(j, F(i) \bigr) . \]
The sheaf property says that these glue to an isomorphism
$e \in \gerbe{H}(X) \bigl( j, F(i) \bigr)$.
\end{proof}

\begin{rem}
If we were to use open hypercoverings in Construction \ref{const1},
then the obstruction class $\opn{cl}^2_F(j)$ would always be defined,
as an element of $\mrm{H}^2(X, \mcal{N})$. However the technicalities
involved in proving the corresponding version of Theorem \ref{thm.2}
would be enormous. Since Construction \ref{const1} works in the cases
that interest us, we chose to limit ourselves to this weaker
approach.
\end{rem}

\begin{rem}
L. Breen [private communication] proposed looking at the
the central extension of gerbes (\ref{eqn:5}) in the following way:
$\gerbe{G}$ is a gerbe over $\gerbe{H}$, with band $\mcal{N}$. 
Perhaps this point of view can yield a stronger version of
Theorem \ref{thm.2}.
\end{rem}

%\cleardoublepage
\section{Sufficient Conditions for Existence of Obstruction Classes}
\label{sec.suff}

Let $\mcal{N}$ be a sheaf of abelian groups on a topological space
$X$. The operation in $\mcal{N}$ is multiplication.
We denote by $\mrm{H}^i (X, \mcal{N})$ the derived functor sheaf
cohomology. An open
set $U \subset X$ will be called {\em $\mcal{N}$-acyclic} if the sheaf
cohomology satisfies
$\mrm{H}^i (U, \mcal{N}) = 1$ for all $i > 0$. Now suppose 
$\bsym{U} = \{ U_k \}_{k \in K}$ is a collection of open sets in $X$. 
We say that the collection $\bsym{U}$ is $\mcal{N}$-acyclic if all the open 
sets $U_{k_0, \ldots, k_m}$ are $\mcal{N}$-acyclic.

\begin{dfn}
Let $\mcal{N}$ be a sheaf of abelian groups on $X$.
We say that {\em there are enough $\mcal{N}$-acyclic open coverings} if for 
any open set $U \subset X$, and any open covering $\bsym{U}$ of $U$,
there exists an $\mcal{N}$-acyclic open covering $\bsym{U}'$ of $U$ which
refines $\bsym{U}$. 
\end{dfn}

\begin{exa} \label{exa:12}
Suppose $X$ is a differentiable (i.e.\ $\mrm{C}^{\infty}$) manifold, and let
$\mcal{O}_X$ be the sheaf of $\mrm{C}^{\infty}$ $\mbb{R}$-valued functions on
it. If $\mcal{N}$ is an $\mcal{O}_X$-module, then any open covering of
$X$ is $\mcal{N}$-acyclic. If $\mcal{K}$ is a constant sheaf of
abelian groups on $X$, then any open covering $\bsym{U} = \{ U_k \}_{k
\in K}$ such that 
the finite intersections $U_{k_0, \ldots, k_m}$ are contractible, is 
$\mcal{K}$-acyclic. There are always enough coverings of this sort.
\end{exa}

\begin{exa} \label{exa.2}
Suppose $X$ is a complex analytic manifold, and let
$\mcal{O}_X$ be the sheaf of holomorphic $\mbb{C}$-valued functions on
it. If $\mcal{N}$ is a coherent $\mcal{O}_X$-module, then any open covering of
$X$ by Stein manifolds is $\mcal{N}$-acyclic. There are always enough coverings
of this sort. 
Regarding constant sheaves see the previous example. (Oddly, we do not know if 
it is possible to find an open covering 
$\bsym{U} = \{ U_k \}_{k \in K}$ such that 
the finite intersections $U_{k_0, \ldots, k_m}$ are both contractible and
Stein.)
\end{exa}

\begin{exa} \label{exa.3}
Suppose $X$ is an algebraic variety over a field $\K$ (i.e.\ a
separated integral finite type $\K$-scheme), and let
$\mcal{O}_X$ be the structure sheaf. If $\mcal{N}$ is a coherent
$\mcal{O}_X$-module, then any affine open covering of $X$ (i.e.\ a covering
$\bsym{U} = \{ U_k \}_{k \in K}$ such that the open sets $U_k$ are all affine)
is $\mcal{N}$-acyclic. There are always enough coverings of this sort.
If $\mcal{K}$ is a constant sheaf of abelian groups
on $X$, then any open covering of $X$ is
$\mcal{K}$-acyclic (since $\mcal{K}$ is a flasque sheaf in the Zariski
topology).
\end{exa}

Recall that there are canonical group homomorphisms
\[ \check{\mrm{H}}^i(X, \mcal{N}) \to \mrm{H}^i(X, \mcal{N}) , \]
which are bijective for $i = 0, 1$; see \cite[Section III.4.]{Ha}.

\begin{prop} \label{prop:11}
Let $\mcal{N}$ be a sheaf of abelian groups on $X$.
\begin{enumerate}
\item If $\bsym{U}$ is an $\mcal{N}$-acyclic open covering of $X$,
then the canonical group homomorphisms
\[ \check{\mrm{H}}^i(\bsym{U}, \mcal{N}) \to \mrm{H}^i(X, \mcal{N}) \]
are bijective for all $i$. 
\item If there are enough $\mcal{N}$-acyclic open coverings, then for
any $\mcal{N}$-acyclic open covering  $\bsym{U}$ of $X$,
the canonical group homomorphisms
\[ \check{\mrm{H}}^i(\bsym{U}, \mcal{N}) \to
\check{\mrm{H}}^i(X, \mcal{N}) \to \mrm{H}^i(X, \mcal{N}) \]
are bijective for all $i$. 
\end{enumerate}
\end{prop}

\begin{proof}
Assertion (1) is \cite[Exercise III.4.11]{Ha}. Assertion (2) follows
from (1). See also the original \cite{Gr2}.
\end{proof}

 From now on in this section, the operation in the group $\mcal{N}$ is
multiplication, and the identity element is $1$. 

\begin{prop} \label{prop:9}
Suppose
\begin{equation} \label{eqn:15}
1 \to \mcal{N} \to \mcal{G} \to \mcal{H} \to 1 
\end{equation}
is an exact sequence of sheaves of groups on $X$. 
\begin{enumerate}
\item There is an exact sequence in \v{C}ech cohomology
\[ \begin{aligned}
& 1 \to \mcal{N}(X) \to \mcal{G}(X) \to \mcal{H}(X) \\
& \qquad \to \check{\mrm{H}}^1(X, \mcal{N}) \to \check{\mrm{H}}^1(X, \mcal{G}) 
\to \check{\mrm{H}}^1(X, \mcal{H}) .
\end{aligned} \]
Here $\check{\mrm{H}}^1(X, -)$ are pointed sets. 
\item Assume \tup{(\ref{eqn:15})} is a central extension, and there are enough
$\mcal{N}$-acyclic open coverings. Then the exact sequence of part \tup{(1)}
extends to an exact sequence
\[ \cdots \to \check{\mrm{H}}^1(X, \mcal{G}) 
\to \check{\mrm{H}}^1(X, \mcal{H}) \xar{\partial} \check{\mrm{H}}^2(X, \mcal{N})
. \]
\end{enumerate}
\end{prop}

\begin{proof}
(1) This is pretty easy. A readable proof can be found in \cite[Chapter V]{Gr1}.

\medskip \noindent
(2) A more general result is \cite[Corollaire to Proposition 3.4.2]{Gr2},
where there is no topological assumption of the sheaf $\mcal{N}$.
However, the precise statement and the proof rely on Godement
resolutions, and are hard to follow. Hence we
provide a relatively easy proof in the case we need. 

Recall that the pointed set $\check{\mrm{H}}^1(X, \mcal{H})$ classifies left 
$\mcal{H}$-torsors on $X$, up to isomorphism. And the function
$\check{\mrm{H}}^1(X, \mcal{G}) \to \check{\mrm{H}}^1(X, \mcal{H})$
sends a $\mcal{G}$-torsor to the induced $\mcal{H}$-torsor.

Let $\mcal{S}$ be an $\mcal{H}$-torsor. 
Choose an $\mcal{N}$-acyclic open covering 
$\bsym{U} = \{ U_k \}_{k \in K}$ of $X$ that trivializes $\mcal{S}$. 
For any index $k$ choose some
$s_k \in \mcal{S}(U_k)$. For any $k_0, k_1$ we have an element
$h_{k_0, k_1} \in \mcal{H}(U_{k_0, k_1})$ such that
$s_{k_1} = h_{k_0, k_1} \cdot s_{k_0}$.
Since 
$\check{\mrm{H}}^1(U_{k_0, k_1}, \mcal{N}) = 1$, by part (1) we have a
surjection of groups
$\mcal{G}(U_{k_0, k_1}) \to \mcal{H}(U_{k_0, k_1})$,
and thus we can lift $h_{k_0, k_1}$ to some
$g_{k_0, k_1} \in \mcal{G}(U_{k_0, k_1})$.
Define
\[ n_{k_0, k_1, k_2} := g_{k_0, k_2}^{-1} \circ g_{k_1, k_2} \circ
g_{k_0, k_1} \in \mcal{G}(U_{k_0, k_1, k_2}) . \]
Then 
\[ c := \{ n_{k_0, k_1, k_2} \}_{k_0, k_1, k_2 \in K} \]
is a \v{C}ech $2$-cocycle with values in $\mcal{N}$; cf.\
Lemma \ref{lem:10}. Let
\[ \partial(\mcal{S}) := [c] \in \check{\mrm{H}}^2(X, \mcal{N}) . \]
As in the proof of Lemma \ref{lem.4} we see that the cohomology class 
$\partial(\mcal{S})$ is independent of choices, and thus we get a well defined
function
\[ \partial : \check{\mrm{H}}^1(X, \mcal{H}) \to \check{\mrm{H}}^2(X, \mcal{N})
. \]
And like in the proof of Theorem \ref{thm.2} we see that
$\partial(\mcal{S}) = 1$ if and only if $\mcal{S}$ comes from a 
$\mcal{G}$-torsor. 
\end{proof}

Consider a central extension of gerbes
\begin{equation} \label{eqn:9}
1 \to  \mcal{N} \to \gerbe{G} \xar{F} \gerbe{H} \to 1 
\end{equation}
on $X$.

\begin{lem} \label{lem.A.1}
Suppose $U$ is an $\mcal{N}$-acyclic open set. Let
$i, j \in \opn{Ob} \gerbe{G}(U)$ be such that 
$\gerbe{G}(U)(i, j) \neq \emptyset$. Then the function
\[ F : \gerbe{G}(U)(i, j) \to 
\gerbe{H}(U) \bigl( F(i), F(j) \bigr) \]
is surjective.
\end{lem}

\begin{proof}
Here both torsors 
$\gerbe{G}(i, j)$ and $\gerbe{H} \bigl( F(i), F(j) \bigr)$
are trivial over the respective sheaves of groups 
$\gerbe{G}(i, i)$ and $\gerbe{H} \bigl( F(i), F(i) \bigr)$;
so we may assume $i = j$. Since
$\check{\mrm{H}}^1(U, \mcal{N}) = 1$ the assertion follows from the exact
sequence in Proposition \ref{prop:9}(1), applied to the short exact sequence of
sheaves of groups
\[ 1 \to \mcal{N}|_U \to \gerbe{G}(i,i) \xar{F} \gerbe{H}(i,i) \to
1 . \]
\end{proof}

\begin{lem} \label{lem.A.2}
Suppose $U$ is an $\mcal{N}$-acyclic open set. Then for any
$i, j \in \opn{Ob} \gerbe{G}(U)$ the function
\[ F(i, j) : \gerbe{G}(U)(i, j) \to 
\gerbe{H}(U) \bigl( F(i), F(j) \bigr) \]
is surjective.
\end{lem}

\begin{proof}
If $\gerbe{H}(U) \bigl( F(i), F(j) \bigr) = \emptyset$ then there is
nothing to prove. So let us assume it is nonempty. We will prove that 
$\gerbe{G}(U) \bigl( i, j \bigr) \neq \emptyset$; and then the assertion
will follow by Lemma \ref{lem.A.1}.

Choose some $h \in \gerbe{H}(U) \bigl( F(i), F(j) \bigr)$. 
Let $\bsym{U} = \{ U_k \}_{k \in K}$ be an open covering of $U$, such that 
for any $k$ there exists an isomorphism 
$g_k \in \gerbe{G}(U_k)(i, j)$ lifting $h$.
This can be done. 
Now for $k_0, k_1 \in K$ define
\[ g_{k_0, k_1} := g_{k_1}^{-1} \circ g_{k_0} \in 
\gerbe{G}(U_{k_0, k_1})(i, j) . \]
Since $F(g_{k_0, k_1}) = 1$ we see that in fact
\[ g_{k_0, k_1} \in \mcal{N}(U_{k_0, k_1}) . \]
An easy calculation shows that the \v{C}ech $1$-cochain 
$\{ g_{k_0, k_1} \}_{k_0, k_1 \in K}$ is a cocycle. 
Since $\check{\mrm{H}}^1(U, \mcal{N}) = 1$, after possibly replacing
$\bsym{U}$ with a refinement, we can find a $0$-cochain
$\{ f_k \}_{k \in K}$ such that 
$g_{k_0, k_1} = f_{k_1}^{-1} \circ f_{k_0}$. 
Define
\[ g'_{k} := g_k \circ f_k^{-1} \in \gerbe{G}(U_{k})(i, j) . \]
Then the $0$-cochain $\{ g'_{k} \}_{k \in K}$ is a cocycle with values in the
sheaf of sets $\gerbe{G}(i, j)$. 
 From the sheaf property it follows that there is an element
$g' \in \gerbe{G}(U)(i, j)$ 
such that $g'|_{U_k} = g'_k$ for all $k$. We see that
$\gerbe{G}(U) \bigl( i, j \bigr) \neq \emptyset$.
\end{proof}

\begin{thm} \label{thm:10}
Consider the central extension of gerbes \tup{(\ref{eqn:9})}. If there are
enough $\mcal{N}$-acyclic open coverings, then the obstruction class
$\opn{cl}^2_F(j)$ from Definition \tup{\ref{dfn.4}} exists, for any
$j \in \opn{Ob} \gerbe{H}(X)$.
\end{thm}

\begin{proof}
Since the morphism of gerbes $F$ is locally surjective on objects, we
can find an open covering $\bsym{U} = \{ U_k \}_{k \in K}$ of $X$,
and objects $i_k \in \opn{Ob} \gerbe{G}(U_k)$ that lift
$j|_{U_k}$. By refining it we can assume that $\bsym{U}$ is
$\mcal{N}$-acyclic. According to Lemma \ref{lem.A.2} 
there exist elements $g_{k_0, k_1}$ that lift the elements 
$h_{k_0, k_1}$, in the notation of Construction \ref{const1}.
\end{proof}

Now suppose we are given a morphism of central extensions of gerbes
\begin{equation} \label{eqn:10}
\UseTips \xymatrix @C=5ex @R=5ex {
1 
\ar[r]
&
\mcal{N}
\ar[r]
\ar[d]
&
\gerbe{G}
\ar[r]^{F}
\ar[d]_{D}
& 
\gerbe{H}
\ar[r]
\ar[d]_{E}
& 
1
\\
1 
\ar[r]
&
\mcal{N}'
\ar[r]
&
\gerbe{G}'
\ar[r]^{F'}
& 
\gerbe{H}'
\ar[r]
& 
1
} 
\end{equation}
There is a  homomorphism of sheaves of abelian groups
$D : \mcal{N} \to \mcal{N}'$, and an induced homomorphism
\[ D : \check{\mrm{H}}^2(X, \mcal{N}) \to \check{\mrm{H}}^2(X, \mcal{N}') . \]

\begin{prop} \label{prop:7}
Consider the morphism of central extension of gerbes \tup{(\ref{eqn:10})}. 
\begin{enumerate}
\item Let $j \in \opn{Ob} \gerbe{H}(X)$ be such that
the obstruction class $\opn{cl}^2_F(j)$ is defined, and let 
$j' := E(j) \in \opn{Ob} \gerbe{H}'(X)$. 
Then the obstruction class $\opn{cl}^2_{F'}(j')$ is also defined, and moreover
\[ \opn{cl}^2_{F'}(j') = D \bigl( \opn{cl}^2_{F}(j) \bigr)  \]
in $\check{\mrm{H}}^2(X, \mcal{N}')$.
\item Let $i, j \in \opn{Ob} \gerbe{G}(X)$ and let
$h \in \gerbe{H}(X) \bigl( F(i), F(j) \bigr)$. 
We write $i' := D(i)$ and $j' := D(j)$ for the corresponding objects of
$\gerbe{G}'(X)$, and
$h' := E(h)$ for the corresponding isomorphism $i' \to j'$. Then 
\[ \opn{cl}^1_{F'}(h') = D \bigl( \opn{cl}^1_{F}(h) \bigr)  \]
in $\check{\mrm{H}}^1(X, \mcal{N}')$.
\end{enumerate}
\end{prop}

\begin{proof}
Take the choices made in constructing the class
$\opn{cl}^2_{F}(j)$ or $\opn{cl}^1_{F}(h)$, as the case may be, and use the same
open covering, and the images under $D, F$ of the elements, to construct the
class $\opn{cl}^2_{F'}(j')$ or $\opn{cl}^1_{F'}(h')$.
\end{proof}

\begin{cor}
Consider the morphism of central extension of gerbes \tup{(\ref{eqn:10})}.
Assume that $E$ is an equivalence, and that there are enough  $\mcal{N}$-acyclic
open coverings. Then the obstruction class $\opn{cl}^2_{F'}(j')$ is defined for
any $j' \in \opn{Ob} \gerbe{H}'(X)$.
\end{cor}

\begin{proof}
There is some $j \in \opn{Ob} \gerbe{H}(X)$ such that
$\gerbe{H}'(X) \bigl( j', E(j) \bigr) \neq \emptyset$.
Now use Propositions \ref{prop:7}(1) and \ref{prop:10}.
\end{proof}

%\cleardoublepage
\section{Pronilpotent Gerbes}
\label{sec:pronilp}

Let $X$ be a topological space. Recall that given an inverse system 
$\{ \mcal{G}_p \}_{p \in \N}$ of sheaves of groups on $X$, its inverse limit is
the sheaf of groups 
$\opn{lim}_{\leftarrow p} \, \mcal{G}_p$
whose group of sections on an open set $U$ is
\[ \Gamma(U, \opn{lim}_{\leftarrow p} \, \mcal{G}_p) = 
\opn{lim}_{\leftarrow p} \, \Gamma(U, \mcal{G}_p) . \]

\begin{dfn} \label{dfn.10}
Let $\mcal{G}$ be a sheaf of groups on $X$. 
\begin{enumerate}
\item A {\em normal filtration} of $\mcal{G}$ is a descending sequence 
$\{ \mcal{N}_p \}_{p \in \N}$ of sheaves of normal subgroups of $\mcal{G}$.
\item A {\em central filtration} of $\mcal{G}$ is a normal filtration
$\{ \mcal{N}_p \}_{p \in \N}$, such that 
$\mcal{N}_0 = \mcal{G}$, $\bigcap_p \mcal{N}_p = 1$, and for every $p$ the
extension of sheaves of groups 
\[ 1 \to \mcal{N}_p / \mcal{N}_{p+1} \to \mcal{G} / \mcal{N}_{p+1} 
\to \mcal{G} / \mcal{N}_p \to 1 \]
is central.
\item Let $\{ \mcal{N}_p \}_{p \in \N}$ be a normal filtration of $\mcal{G}$.
We say that $\mcal{G}$ is {\em complete} with respect to this filtration if
the canonical homomorphism of sheaves of groups
\[ \mcal{G} \to \opn{lim}_{\leftarrow p} \, 
\mcal{G} / \mcal{N}_{p} \] 
is an isomorphism.
\item If $\mcal{G}$ is complete with respect to some central filtration,
then we call it a {\em pronilpotent sheaf of groups}.
\end{enumerate}
\end{dfn}

Note that when $\{ \mcal{N}_p \}_{p \in \N}$ is a central filtration, then
each $\mcal{N}_p / \mcal{N}_{p+1}$ is a sheaf of abelian groups.

\begin{dfn} \label{dfn:11}
Let $\mcal{G}$ be a sheaf of groups on $X$, with central filtration
$\{ \mcal{N}_p \}_{p \in \N}$. An open set $U \subset X$ is called 
{\em acyclic} with respect to  $\{ \mcal{N}_p \}_{p \in \N}$ if
the following two conditions hold.
\begin{enumerate}
\rmitem{i} For every $p \geq 0$ and $i > 0$ the sheaf cohomology group 
$\mrm{H}^i(U, \mcal{N}_p / \mcal{N}_{p+1})$ is trivial.
\rmitem{ii} For every $q \geq p \geq 0$ the canonical group homomorphism
\[ \Gamma(U, \mcal{N}_p) \to \Gamma(U, \mcal{N}_p / \mcal{N}_{q}) \] 
is surjective. 
\end{enumerate}
\end{dfn}

\begin{lem} \label{lem.1}
Let $\{ \mcal{N}_p \}_{p \in \N}$ be a central filtration of the sheaf of
groups $\mcal{G}$. Suppose that $\mcal{G}$ is complete with respect to the
filtration $\{ \mcal{N}_p \}_{p \in \N}$, and the open set $U \subset X$ is 
acyclic with respect to $\{ \mcal{N}_p \}_{p \in \N}$. 
Let $G := \Gamma(U, \mcal{G})$ and
$N_p := \Gamma(U, \mcal{N}_p)$.
Then $G$ is complete with respect to the filtration $\{ N_p \}_{p \in \N}$.
\end{lem}

\begin{proof}
Condition (ii) of Definition \ref{dfn:11}, combined with Proposition
\ref{prop:9}(1), say that for every $p$ there is an exact sequence of groups
\[ 1 \to N_p \to G \to \Gamma(U, \mcal{G} / \mcal{N}_{p}) \to 1 . \]
Now use Definition \ref{dfn.10}(3).
\end{proof}

In the situation above the filtration $\{ N_p \}_{p \in \N}$ of $G$ is
separated. Therefore it defines a metric topology on
$G$, say by letting $N_p \cdot g = g \cdot N_p$ be the ball of radius $2^{-p}$
around the point $g \in G$ (cf.\ \cite[Section III.5]{CA}).
The condition that $G \cong \opn{lim}_{\leftarrow p} \, G / N_p$
translates to $G$ being a complete metric space.

\begin{dfn}
Let $\mcal{G}$ be a sheaf of groups on $X$, with central filtration
$\{ \mcal{N}_p \}_{p \in \N}$.
\begin{enumerate}
\item A collection $\bsym{U} = \{ U_k \}_{k \in K}$ of open sets of $X$ is
called {\em  acyclic} with respect to  $\{ \mcal{N}_p \}_{p \in \N}$
if every finite intersection $U_{k_0, \ldots, k_m}$ is acyclic with respect to 
$\{ \mcal{N}_p \}_{p \in \N}$, in the sense of Definition \ref{dfn:11}.
\item We say that there are {\em enough acyclic coverings} with respect to 
$\{ \mcal{N}_p \}_{p \in \N}$ if every open covering $\bsym{U}$ of an open set
$U \subset X$ admits a refinement $\bsym{U}'$ which is acyclic  with
respect to $\{ \mcal{N}_p \}_{p \in \N}$.
\end{enumerate}
\end{dfn}

Now we move to gerbes. Let $\gerbe{G}$ be a gerbe on $X$.
The notion of normal subgroupoid $\gerbe{N} \subset \gerbe{G}$ was introduced
in Definition \ref{dfn.9}.

\begin{dfn} \label{dfn.12}
Let $\gerbe{G}$ be a gerbe on $X$. 
\begin{enumerate}
\item A {\em normal filtration} of $\gerbe{G}$ is a descending sequence 
$\{ \gerbe{N}_p \}_{p \in \N}$ 
of normal subgroupoids of $\gerbe{G}$.
\item A {\em central filtration} of $\gerbe{G}$
is a normal filtration $\{ \gerbe{N}_p \}_{p \in \N}$, such that for every
local object $i$ of $\gerbe{G}$, the filtration 
$\{ \gerbe{N}_p(i, i) \}_{p \in \N}$ of the sheaf of groups 
$\gerbe{G}(i, i)$ is central.
\item Let $\{ \gerbe{N}_p \}_{p \in \N}$ be a normal filtration of $\gerbe{G}$.
We say that $\gerbe{G}$ is {\em complete} with respect to this filtration if
for every local object $i$ of $\gerbe{G}$, the sheaf $\gerbe{G}(i, i)$
is complete with respect to the filtration 
 $\{ \gerbe{N}_p(i, i) \}_{p \in \N}$.
\item If $\gerbe{G}$ is complete with respect to some central filtration,
then we call it a {\em pronilpotent gerbe}.
\end{enumerate}
\end{dfn}

\begin{prop}
Suppose $\{ \gerbe{N}_p \}_{p \in \N}$ is a central filtration of the gerbe
$\gerbe{G}$. Then for every $p$ there is a central extension of gerbes
\begin{equation} \label{eqn:16}
1 \to \gerbe{N}_p / \gerbe{N}_{p+1} \to 
\gerbe{G} / \gerbe{N}_{p+1} \xar{F} \gerbe{G} / \gerbe{N}_{p} \to 1 .
\end{equation}
\end{prop}

\begin{proof}
Use Theorem \ref{thm:1}.
\end{proof}

Observe that $\gerbe{N}_p / \gerbe{N}_{p+1}$ is a central subgroupoid of the
gerbe $\gerbe{G} / \gerbe{N}_{p+1}$; so 
$\gerbe{N}_p / \gerbe{N}_{p+1}$ can be
regarded as a sheaf of abelian groups on $X$. See Proposition \ref{prop:12}.

\begin{dfn}
Let $\gerbe{G}$ be a gerbe on $X$, with central filtration
$\{ \gerbe{N}_p \}_{p \in \N}$.
An open set $U \subset X$ is called {\em acyclic} with respect to
$\{ \gerbe{N}_p \}_{p \in \N}$ if the following two condition hold:
\begin{enumerate}
\rmitem{i} The groupoid $\gerbe{G}(U)$ is nonempty.
\rmitem{ii} For every $i \in \opn{Ob} \gerbe{G}(U)$, the set $U$ is 
acyclic with respect to the central filtration
$\{ \gerbe{N}_p(i, i) \}_{p \in \N}$
of the sheaf of groups $\gerbe{G}(i, i)$. 
\end{enumerate}
\end{dfn}

\begin{dfn}
Let $\gerbe{G}$ be a gerbe on $X$, with central filtration
$\{ \gerbe{N}_p \}_{p \in \N}$.
\begin{enumerate}
\item A collection $\bsym{U} = \{ U_k \}_{k \in K}$ of open sets of $X$ is
called {\em acyclic} with respect to  $\{ \gerbe{N}_p \}_{p \in \N}$
if every finite intersection $U_{k_0, \ldots, k_m}$ is acyclic with respect to 
$\{ \gerbe{N}_p \}_{p \in \N}$.
\item We say that there are {\em enough  acyclic  coverings} with respect to 
$\{ \gerbe{N}_p \}_{p \in \N}$ if every open covering $\bsym{U}$ of an open set
$U$ admits a refinement $\bsym{U}'$ which is acyclic  with
respect to $\{ \gerbe{N}_p \}_{p \in \N}$.
\end{enumerate}
\end{dfn}

Suppose $\gerbe{G}$ is a gerbe, complete with respect to a  central
filtration $\{ \gerbe{N}_p \}_{p \in \N}$. Let $U$ be an open set of $X$ which
is acyclic with respect to $\{ \gerbe{N}_p \}_{p \in \N}$, and
let $i$ be an object of the groupoid $\gerbe{G}(U)$. 
By Lemma \ref{lem.1} the group $\gerbe{G}(U)(i, i)$ is complete with respect to
the filtration $\{ \gerbe{N}_p(U)(i, i) \}_{p \in \N}$; so 
$\gerbe{G}(U)(i, i)$ is a complete metric space.
Now let $j$ be another object of $\gerbe{G}(U)$, and suppose 
$\gerbe{G}(U)(i, j) \neq \emptyset$. Then the set
$\gerbe{G}(U)(i, j)$ is a $\gerbe{G}(U)(j, j)$-$\gerbe{G}(U)(i, i)$-bitorsor.
One can introduce a metric topology on this set, by letting 
\[ g \circ \gerbe{N}_p(U)(i, i) = \gerbe{N}_p(U)(j, j) \circ g \]
be the ball of radius
$2^{-p}$ around $g \in \gerbe{G}(U)(i, j)$. 
For any such $g$ the function 
$h \mapsto g \cdot h$ is an isomorphism of metric spaces
$\gerbe{G}(U)(i, i) \iso \gerbe{G}(U)(i, j)$; and hence 
$\gerbe{G}(U)(i, j)$ is complete.

\begin{thm} \label{thm.4}
Let $\gerbe{G}$ be a gerbe on the topological space $X$, and let 
$\{ \gerbe{N}_p \}_{p \in \N}$ be a central filtration on it.
Assume that $\gerbe{G}$ is complete with respect to 
$\{ \gerbe{N}_p \}_{p \in \N}$, and that there are enough
acyclic coverings with respect to 
$\{ \gerbe{N}_p \}_{p \in \N}$.
Let $U$ be some open set of $X$.
\begin{enumerate}
\item If $\mrm{H}^1(U, \gerbe{N}_p / \gerbe{N}_{p+1}) = 1$
for every $p \geq 0$, then the groupoid $\gerbe{G}(U)$ is connected.
\item If $\mrm{H}^2(U, \gerbe{N}_p / \gerbe{N}_{p+1}) = 1$
for every $p \geq 0$, then the groupoid $\gerbe{G}(U)$ is nonempty.
\end{enumerate}
\end{thm}

\begin{proof}
(1) This is very similar to Theorem \ref{thm.3}. Given 
$i, j \in \opn{Ob} \gerbe{G}(U)$, we must show that 
$\gerbe{G}(U)(i, j) \neq \emptyset$.

Since $\gerbe{G}$ is locally connected, we can find an open covering 
$\bsym{U} = \{ U_k \}_{k \in K}$ of $U$ such that
$\gerbe{G}(U_k)(i, j) \neq \emptyset$ for any $k \in K$.
By refining $\bsym{U}$, we can assume that it is  acyclic with
respect to $\{ \gerbe{N}_p \}_{p \in \N}$.
For each $k \in K$ let us choose an element
$g_{k; 0} \in \gerbe{G}(U_k)(i, j)$.
We are going to construct new elements
$g_{k; p} \in \gerbe{G}(U_k)(i, j)$, for all $k \in K$ and $p \in \N$, 
satisfying these conditions:
\begin{enumerate}
\rmitem{a} $g_{k; p+1} \in g_{k; p} \circ \gerbe{N}_p(U_k)(i, i)$.
\rmitem{b} $g_{k_1; p}^{-1} \circ g_{k_0; p} \in 
\gerbe{N}_p(U_{k_0, k_1})(i, i)$
for any $k_0, k_1 \in K$.
\end{enumerate}
The construction is by recursion on $p$.

For $p = 0$ the elements $g_{k; 0}$ are already given. So let $p \geq 0$,
and assume that we have elements $g_{k; p'}$ for $p' \leq p$, satisfying
conditions (a)-(b). 
Let us denote by 
$\bar{g}_{k; p} \in (\gerbe{G} / \gerbe{N}_{p+1})(U_k)(i, j)$
the image of $g_{k; p}$, and define
\[ \bar{g}_{k_0, k_1; p} := \bar{g}_{k_1; p}^{-1} \circ
\bar{g}_{k_0; p} \in (\gerbe{G} / \gerbe{N}_{p+1})(U_{k_0, k_1})(i, i)  \]
for $k_0, k_1 \in K$.
Consider the central extension of gerbes (\ref{eqn:16}).
By condition (b) we have 
$F(\bar{g}_{k_0, k_1; p}) = 1$; 
so
\[ \bar{g}_{k_0, k_1; p} \in 
(\gerbe{N}_p / \gerbe{N}_{p+1})(U_{k_0, k_1})(i, i) . \]
We get a \v{C}ech $1$-cocycle 
$c := \{ \bar{g}_{k_0, k_1; p} \}_{k_0, k_1 \in K}$ with values in the sheaf
$\gerbe{N}_p / \gerbe{N}_{p+1}$.

According to the assumptions and Proposition \ref{prop:11}(1), we have
\[ \check{\mrm{H}}^1(\bsym{U}, \gerbe{N}_p / \gerbe{N}_{p+1}) \cong 
\mrm{H}^1(U, \gerbe{N}_p / \gerbe{N}_{p+1}) = 1 . \]
Hence there is a $0$-cochain 
$b = \{ \bar{f}_k \}_{k \in K}$ 
such that $c = \d(b)$, where $\d$ is the \v{C}ech coboundary.
By condition (ii) of Definition \ref{dfn:11} the homomorphism
\[ \gerbe{N}_p(U_k)(i, i) \to (\gerbe{N}_p / \gerbe{N}_{p+1})(U_k)(i, i) \]
is surjective, so we can lift $\bar{f}_k$ to an element
$f_k \in \gerbe{N}_p(U_k)(i)$.
Let us define
\[ g_{k; p+1} := g_{k; p} \cdot f_k^{-1} . \]
Then conditions (a)-(b) are satisfied. 

We know that the set $\gerbe{G}(U_k)(i, j)$ is a complete metric space. 
Condition (a) says that $\{ g_{k; p} \}_{p \in \N}$ is a Cauchy sequence. 
Let 
\[ g_k := \lim_{p \to \infty} \, g_{k; p} \in \gerbe{G}(U_k)(i, j) . \]
Condition (b) now says that $\{ g_k \}_{k \in K}$ is a $1$-cocycle. 
By descent for morphisms there is an element
$g \in \gerbe{G}(U)(i, j)$ such that $g|_{U_k} = g_k$.

\medskip \noindent
(2) This is like Theorem \ref{thm.2}. Since $\gerbe{G}$ is locally nonempty,
there is an open covering $\bsym{U} = \{ U_k \}_{k \in K}$ of $U$ such that 
all the groupoids $\gerbe{G}(U_k)$ are nonempty. By refining $\bsym{U}$ we may
assume it is acyclic with respect to $\{ \gerbe{N}_p \}_{p \in \N}$.
Let's choose some $i_k \in \opn{Ob} \gerbe{G}(U_k)$.
For any $k_0, k_1 \in K$, and any $p \in \N$, we have
$\mrm{H}^1(U_{k_0, k_1}, \gerbe{N}_p / \gerbe{N}_{p+1}) = 1$. According to part
(1) of the theorem, applied to the open set $U_{k_0, k_1}$, the groupoid 
$\gerbe{G}(U_{k_0, k_1})$ is connected. 
Let us choose some element
$g_{k_0, k_1; 0} \in \gerbe{G}(U_{k_0, k_1})(i_{k_0}, i_{k_1})$.

Using recursion on $p$ we shall construct elements
$g_{k_0, k_1; p} \in \gerbe{G}(U_{k_0, k_1})(i_{k_0}, i_{k_1})$
satisfying these conditions:
\begin{enumerate}
\rmitem{a} $g_{k_0, k_1; p+1} \in g_{k_0, k_1; p} 
\circ \gerbe{N}_p(U_{k_0, k_1})(i_{k_0}, i_{k_0})$.
\rmitem{b} $g_{k_0, k_2; p}^{-1} \circ g_{k_1, k_2; p} \circ
g_{k_0, k_1; p} \in \gerbe{N}_p(U_{k_0, k_1, k_2})(i_{k_0}, i_{k_0})$
for any $k_0, k_1, k_2 \in K$.
\end{enumerate}

For $p = 0$ the elements $g_{k_0, k_1; 0}$ are already given. So let
$p \geq 0$,
and assume that we have elements $g_{k_0, k_1; p'}$ for $p' \leq p$, satisfying
conditions (a)-(b). 
Let us denote by 
\[ \bar{g}_{k_0, k_1; p} \in (\gerbe{G} / \gerbe{N}_{p+1})(U_{k_0, k_1})
(i_{k_0}, i_{k_1}) \]
the image of $g_{k_0, k_1; p}$, and define 
\[ \bar{g}_{k_0, k_1, k_2; p} := 
\bar{g}_{k_0, k_2; p}^{-1} \circ \bar{g}_{k_1, k_2; p} \circ
\bar{g}_{k_0, k_1; p} \in 
(\gerbe{G} / \gerbe{N}_{p+1})(U_{k_0, k_1, k_2})(i_{k_0}, i_{k_0}) . \]
Consider the central extension of gerbes (\ref{eqn:16}).
By condition (b) we have \linebreak
$F(\bar{g}_{k_0, k_1, k_2; p}) = 1$; 
so
\[ \bar{g}_{k_0, k_1, k_2; p} \in 
(\gerbe{N}_p / \gerbe{N}_{p+1})(U_{k_0, k_1, k_2})(i_{k_0}, i_{k_0}) . \]
Lemma \ref{lem:10} says that 
$c := \{ \bar{g}_{k_0, k_1, k_2; p} \}$
is a \v{C}ech $2$-cocycle.

According to the assumptions and Proposition \ref{prop:11}(1), we have
\[ \check{\mrm{H}}^2(\bsym{U}, \gerbe{N}_p / \gerbe{N}_{p+1}) \cong 
\mrm{H}^2(U, \gerbe{N}_p / \gerbe{N}_{p+1}) = 1 . \]
Hence there is a $1$-cochain 
$b = \{ \bar{f}_{k_0, k_1} \}$ 
such that $c = \d(b)$. As before, we can lift $\bar{f}_{k_0, k_1}$ to an element
$f_{k_0, k_1} \in \gerbe{N}_p(U_{k_0, k_1})(i_{k_0}, i_{k_0})$.
Let us define
\[ g_{k_0, k_1; p+1} := g_{k_0, k_1; p} \circ f_{k_0, k_1}^{-1} . \]
Then conditions (a)-(b) are satisfied. 

As in the proof of part (1), let
\[ g_{k_0, k_1} := \lim_{p \to \infty} \, g_{k_0, k_1; p} \in 
\gerbe{G}(U_{k_0, k_1})(i_{k_0}, i_{k_1}) . \]
Condition (b) says that $\{ g_{k_0. k_1} \}_{k_0, k_1 \in K}$ is a $2$-cocycle. 
By descent for objects there is an object
$i \in \opn{Ob} \gerbe{G}(U)$.
\end{proof}

%\cleardoublepage
\section{Fake Global Objects of Gerbes}
\label{sec.fake}

In this section $X$ is some topological space. We will study a gerbe
$\gerbe{G}$ on $X$, with center $\mrm{Z}(\gerbe{G})$,
and the central extension of gerbes
\begin{equation} \label{eqn:11}
1 \to  \mrm{Z}(\gerbe{G}) \to \gerbe{G} \xar{F} 
\gerbe{G} / \mrm{Z}(\gerbe{G}) \to 1 .
\end{equation}

\begin{dfn}
An object 
$i \in \opn{Ob}\, \bigl( \gerbe{G} / \mrm{Z}(\gerbe{G}) \bigr)(X)$ 
is called a {\em fake global object of $\gerbe{G}$}.
\end{dfn}

When we need to emphasize that $i \in \opn{Ob} \gerbe{G}(X)$,
as opposed to being in \linebreak
$\opn{Ob}\, \bigl( \gerbe{G} / \mrm{Z}(\gerbe{G}) \bigr)(X)$,
we will say that $i$ is a {\em true global object of $\gerbe{G}$}.

Note that some fake global objects $i$ of $\gerbe{G}$ will lift to
true global objects of $\gerbe{G}$, whereas other won't; this is
determined by the vanishing of the obstruction class 
\[ \opn{cl}^2_F(i) \in \check{\mrm{H}}^2 \bigl( X, \mrm{Z}(\gerbe{G}) 
\bigr) \]
for the central extension of gerbes (\ref{eqn:11}), if this obstruction
class is defined. 

Here is an easy example of a fake global object that does not lift.

\begin{exa} \label{exa:11}
Suppose $X$ is an algebraic variety over a field, with 
$\mrm{H}^2(X, \mcal{O}_X) \neq 0$. 
Choose a nonzero cohomology class $c \in \mrm{H}^2(X, \mcal{O}_X)$. 
There is an abelian gerbe $\gerbe{G}$ corresponding to $c$, and it has
no global objects. (This construction is standard; cf.\ Example \ref{exa:17}
below.) Indeed, here the gerbe 
$\gerbe{G} / \mrm{Z}(\gerbe{G})$ 
is equivalent to the gerbe $\gerbe{I}$ from Example \ref{exa:15}, and
hence it has one global object (up to isomorphism), say $j$. 
We have a central extension of gerbes
\[ 1 \to  \mcal{O}_X \to \gerbe{G} \xar{F} 
\gerbe{G} / \mrm{Z}(\gerbe{G}) \to 1 , \]
and the obstruction class to lifting $j$ to an object of 
$\gerbe{G}(X)$ is $\opn{cl}^2_F(j) = c$. We see that $j$ is a fake global object
of $\gerbe{G}$, which does not lift to a true global object of
$\gerbe{G}$.
\end{exa}

\begin{rem} \label{rem.1}
The reason we are interested in fake global objects has to do
with {\em twisted deformations}. Let $\K$ be a field of
characteristic $0$, and consider the ring of formal power series
$\K[[\hbar]]$ in the variable $\hbar$. Let $(X, \mcal{O}_X)$ be a
ringed space over $\K$, as in Examples \ref{exa:12}-\ref{exa.3}. As explained in
\cite{Ye}, a twisted (Poisson or associative)
$\K[[\hbar]]$-deformation $\gerbe{A}$ of $\mcal{O}_X$ is made up
of many locally defined sheaves of (Poisson or associative)
$\K[[\hbar]]$-algebras
$\mcal{A}_i$, that are glued together by isomorphisms 
$\opn{Ad}(g) : \mcal{A}_i \iso \mcal{A}_j$,
called {\em gauge equivalences}. The indices 
$i,j, \ldots$ are local objects of the {\em gauge gerbe} $\gerbe{G}$
of $\gerbe{A}$, and the isomorphisms $g$ are local sections
of the bitorsors $\gerbe{G}(i, j)$.
The group $\gerbe{G}(i, i)$, for a local object $i$,
is by definition $\opn{exp}(\hbar \mcal{A}_i)$, where $\mcal{A}_i$ is
viewed as
a pronilpotent $\K[[\hbar]]$-linear Lie algebra, with Lie bracket
being either its
Poisson bracket or the commutator of the associative multiplication.
The deformations $\mcal{A}_i$, for $i \in \opn{Ob} \gerbe{G}(X)$,
are called {\em global deformations
belonging to $\gerbe{A}$}. In case such global deformations do
not exist (i.e.\ the gerbe $\gerbe{G}$ is nontrivial), then we say
$\gerbe{A}$ is {\em really twisted}. Note that the gerbe $\gerbe{G}$
is pronilpotent.

Now consider a global deformation $\mcal{A}$ of the following sort: there is an
open covering $X = \bigcup_{k \in K} U_k$, objects
$i_k \in \opn{Ob} \gerbe{G}(U_k)$, and gauge equivalences
$h_k : \mcal{A}|_{U_k} \iso \mcal{A}_{i_k}$, 
such that
\[ h_{k_1} \circ h_{k_0}^{-1} = \opn{Ad}(g_{k_0, k_1})  \]
for some 
$g_{k_0, k_1} \in \gerbe{G}(U_k)(i, i)$. 
So as gauge equivalences 
$\mcal{A}_{i_{k_0}} \iso \mcal{A}_{i_{k_2}}$
we have the equality
\[ \opn{Ad}(g_{k_0, k_2}) = \opn{Ad}(g_{k_1, k_2}) \circ 
\opn{Ad}(g_{k_0, k_1})  . \]
The local isomorphisms $g$ in the center 
$\mrm{Z}(\gerbe{G}(i, i))$ are precisely those such that the 
gauge equivalences $\opn{Ad}(g)$
are trivial. Hence, going to the extension of gerbes (\ref{eqn:11}), we have
\[  F(g_{k_0, k_2}) = F(g_{k_1, k_2}) \circ 
F(g_{k_0, k_1}) \]
in the gerbe $\gerbe{G} / \mrm{Z}(\gerbe{G})$. 
This implies that the global deformation $\mcal{A}$ corresponds to an object
$j \in \opn{Ob} \, \bigl( \gerbe{G} / \mrm{Z}(\gerbe{G}) \bigr)(X)$,
i.e.\ to a fake global object of $\gerbe{G}$. Therefore we call it a 
{\em global deformation falsely belonging to $\gerbe{A}$}.
Observe that the obstruction class 
$\opn{cl}^2_F(j) \in \check{\mrm{H}}^2(X, \mrm{Z}(\gerbe{G}))$ 
is represented by the cocycle 
\[  g_{k_0, k_1, k_2}  :=  g_{k_0, k_2}^{-1} \circ g_{k_1, k_2} \circ
g_{k_0, k_1} . \]
If $\opn{cl}^2_F(j) \neq 1$ then the deformation $\mcal{A}$ does not truly
belong to $\gerbe{A}$.
\end{rem}

Next a result. 

\begin{prop} \label{prop:3} 
Let $\gerbe{G}$ be a gerbe on $X$, and assume there are enough
$\mrm{Z}(\gerbe{G})$-acyclic open coverings.
Let 
$\bar{\gerbe{G}} := \gerbe{G} / \mrm{Z}(\gerbe{G})$.
\begin{enumerate}
\item If $\check{\mrm{H}}^2(X, \mrm{Z}(\gerbe{G})) = 1$, then the
canonical morphism of groupoids
$\gerbe{G}(X) \to \bar{\gerbe{G}}(X)$
is essentially surjective on objects. In particular any fake global
object of $\gerbe{G}$ lifts to a true global object.
\item If moreover $\check{\mrm{H}}^1(X, \mrm{Z}(\gerbe{G})) = 1$, then 
$\gerbe{G}(X) \to \bar{\gerbe{G}}(X)$
is bijective on isomorphism classes of objects.
\end{enumerate}
\end{prop}

\begin{proof}
By Theorem \ref{thm:10} the obstruction classes 
$\opn{cl}^1_F(h)$ and $\opn{cl}^2_F(j)$ are all defined. 
Assertion (1) is a consequence of Theorem \ref{thm.2}, and assertion (2) 
is a consequence of Theorem \ref{thm.3}.
\end{proof}

\begin{exa} \label{exa.5}
Let $\K$ be a smooth algebraic variety over a field $\K$ of characteristic $0$. 
Suppose $\gerbe{A}$ is a twisted (Poisson or associative)
$\K[[\hbar]]$-deformation of $\mcal{O}_X$ which is symplectic. This means that
the first order bracket
$\{ -,- \}_{\gerbe{A}}$ on $\mcal{O}_X$ is nondegenerate 
(cf.\ \cite{Ye}). 
It follows that the center of the gauge gerbe $\gerbe{G}$ is isomorphic
(canonically) to the constant sheaf $\K[[\hbar]]$. 
Now in the Zariski topology constant sheaves have no higher cohomologies; and
hence Proposition \ref{prop:3} applies. So there are as many global
deformations falsely belonging to $\gerbe{A}$ there are global
deformations truly belonging to $\gerbe{A}$ in this case. 
\end{exa}

\begin{exa} \label{exa:17}
Let $\K$ be a smooth algebraic variety over a field $\K$ of
characteristic $0$, and assume there is a nonzero class 
$c \in \mrm{H}^2(X, \mcal{O}_X)$. Then there is a twisted
associative $\K[[\hbar]]$-deformation $\gerbe{A}$ of $\mcal{O}_X$,
which is commutative (i.e.\ each of the local deformations 
$\mcal{A}_i$ belonging to $\gerbe{A}$ is commutative). See 
\cite[Example 6.17]{Ye} for this construction. The gauge
gerbe here is abelian: 
$\gerbe{G} \cong \opn{exp}(\hbar \mcal{O}_X[[\hbar]])$.
We filter it by
$\gerbe{N}_p := 
\opn{exp}(\hbar^{p+1} \mcal{O}_X[[\hbar]])$.
For $p = 0$ we have a central extension of gerbes
\[ 1 \to \mcal{O}_X \xar{\hbar^{}} 
\gerbe{G} / \gerbe{N}_1 \to
\gerbe{G} / \gerbe{N}_0 \to 1 , \]
and for the unique (up to isomorphism) global object $j$ of the gerbe 
$\gerbe{G} / \gerbe{N}_0$, 
the obstruction class is $\opn{cl}^2(j) = c$
(same as in Example \ref{exa:11}). Hence 
$\opn{Ob} (\gerbe{G} / \gerbe{N}_1) = \emptyset$,
implying that 
$\opn{Ob} \gerbe{G}(X) = \emptyset$, so $\gerbe{A}$ is really
twisted. 
\end{exa}

One of the reasons for introducing the obstruction classes
$\opn{cl}^2_F(j)$
is to address the following question. 

\begin{que} \label{que.1}
Does there exist an algebraic variety
$X$, and a {\em symplectic} twisted
$\K[[\hbar]]$-deformation
$\gerbe{A}$ of $\mcal{O}_X$, that is really twisted? We expect
the answer to be positive. Indeed, we think this happens when
$X$ is any abelian surface, and we take any nonzero Poisson bracket
on $\mcal{O}_X$, and let $\gerbe{A}$ be its canonical
quantization (as in \cite[Theorem 0.1]{Ye}), which is a twisted associative
$\K[[\hbar]]$-deformation of $\mcal{O}_X$.
\end{que}

\begin{exa} 
Let $X$ be a complex analytic manifold, and denote by $\mcal{O}_X$ the sheaf
of holomorphic functions. Let
$\gerbe{A}$ be a symplectic (Poisson or associative) twisted
$\mbb{C}[[\hbar]]$-deformation of
$\mcal{O}_X$. Then the gauge gerbe $\gerbe{G}$ has a
central filtration, with
$\gerbe{N}_p(i) = \opn{exp}(\hbar^{p+1} \mcal{A}_i)$
for a local object $i$. 
And the center of $\gerbe{G}$ is
\[ \opn{exp}( \hbar \mbb{C}[[\hbar]] ) =
\opn{exp} \bigl( \prod\nolimits_{m = 1}^{\infty} \mbb{C} \hbar^m \bigr) 
\subset \gerbe{G} . \]
We put on 
$\bar{\gerbe{G}} := \gerbe{G} / \mrm{Z}(\gerbe{G})$
the induced filtration  $\{ \bar{\gerbe{N}}_p \}_{p \in \N}$.
Define a normal subgroupoid
\[ \gerbe{M}_p := \gerbe{N}_{p+1} \cdot
\opn{exp} \bigl( \prod\nolimits_{m = 1}^p \mbb{C} \hbar^m \bigr) \subset 
\gerbe{G} . \]
Then for every $p \geq 0$ there is a 
morphism of central extensions of gerbes
\[ \UseTips \xymatrix @C=5ex @R=5ex {
1 
\ar[r]
&
\mcal{O}_X
\ar[r]
\ar[d]
&
\gerbe{G} / \gerbe{M}_p
\ar[r]
\ar[d]
& 
\bar{\gerbe{G}} / \bar{\gerbe{M}}_p
\ar[r]
\ar[d]^{=}
& 
1
\\
1 
\ar[r]
&
\mcal{O}_X / \mbb{C}
\ar[r]
&
\bar{\gerbe{G}} / \bar{\gerbe{N}}_{p+1}
\ar[r]^{E_p}
& 
\bar{\gerbe{G}} / \bar{\gerbe{N}}_p
\ar[r]
& 
1 .
} \]
We do not know if there are enough acyclic open coverings for the
sheaf $\mcal{O}_X / \mbb{C}$; but according to Proposition
\ref{prop:7},
for any 
$j \in \opn{Ob} (\bar{\gerbe{G}} / \bar{\gerbe{N}}_p)(X)$ 
the obstruction class 
$\opn{cl}^2_{E_n}(j) \in 
\check{\mrm{H}}^2(X, \mcal{O}_X / \mbb{C})$
exists, and moreover it comes from 
$\check{\mrm{H}}^2(X, \mcal{O}_X)$.

Now assume that the homomorphism 
\begin{equation} \label{eqn:13}
\mrm{H}^2(X, \mbb{C}) \to \mrm{H}^2(X, \mcal{O}_X)
\end{equation}
is surjective. This happens when $X$ is the analytification of 
a projective algebraic variety (cf.\ \cite[Section 5]{NT}). 
Then 
$\opn{cl}^2_{E_n}(j) = 1$ for all $p$.
Presumably Theorem \ref{thm.4} can be refined to work with vanishing of
obstruction classes (rather than vanishing of the whole cohomology groups). This
would imply that fake global deformations exist here. 
\end{exa}

\begin{exa}
Suppose $X$ is a smooth projective algebraic variety over $\mbb{C}$ (with the
Zariski topology), and let $X_{\mrm{an}}$ be the corresponding complex
analytic
manifold. Let $\gerbe{A}$ be a symplectic (Poisson or associative)
twisted $\mbb{C}[[\hbar]]$-deformation of $\mcal{O}_X$. There is an induced
deformation $\gerbe{A}_{\mrm{an}}$ of $\mcal{O}_{X_{\mrm{an}}}$.
As in Example \ref{exa.5}, there are as many global deformations truly
belonging to $\gerbe{A}$ as there are global deformations falsely
belonging to it. By the
GAGA principle we have 
$\mrm{H}^p(X, \mcal{O}_X) =  
\mrm{H}^p(X_{\mrm{an}}, \mcal{O}_{X_{\mrm{an}}})$
for all $p$, and therefore there as many global deformations truly belonging to 
$\gerbe{A}$ as there are global deformations truly belonging to
$\gerbe{A}_{\mrm{an}}$. In particular, there might be none
(see Question \ref{que.1}). On the other hand, by Hodge Theory the homomorphism
(\ref{eqn:13}) is surjective here (cf.\ \cite[Section 1.2]{BK}), and therefore
(under the caveat of refining Theorem \ref{thm.4}) there is always some global
deformation falsely belonging to 
$\gerbe{A}_{\mrm{an}}$.
\end{exa}

%\cleardoublepage
%\newpage

\end{document}